\documentclass[11pt]{amsart}
\usepackage{geometry}  
\usepackage{graphicx}
\usepackage{amssymb}
\usepackage{amsmath}
\usepackage{color}
\usepackage{pdfsync}
 \usepackage{tikz}
 \usepackage{tikz-cd}

\usepackage[all]{xy}
\xyoption{matrix}
\xyoption{arrow}

\usetikzlibrary{arrows,decorations.pathmorphing,backgrounds,positioning,fit,petri}

\newtheorem{thm}{Theorem}[section]
\newtheorem{lem}[thm]{Lemma}
\newtheorem{cor}[thm]{Corollary}
\newtheorem{prop}[thm]{Proposition}

\newtheorem{innercustomthm}{Theorem}
\newenvironment{customthm}[1]
    {\renewcommand
    \theinnercustomthm{#1}\innercustomthm}{\endinnercustomthm}

\newtheorem{innercustomlem}{Lemma}
\newenvironment{customlem}[1]
	{\renewcommand
	\theinnercustomlem{#1}\innercustomlem}
	{\endinnercustomlem}

\theoremstyle{definition}
\newtheorem{defn}[thm]{Definition}
\newtheorem{eg}[thm]{Example}

\theoremstyle{remark}
\newtheorem{rem}[thm]{Remark}
 
\numberwithin{equation}{section}

\newcommand{\mat}[1]{\ensuremath{
\left[\begin{matrix}#1
\end{matrix}\right]
}}

 \newcommand{\onto}{\twoheadrightarrow}

\DeclareMathOperator{\Hom}{Hom}%
\DeclareMathOperator{\Ext}{Ext}%
\DeclareMathOperator{\undim}{\underline{dim}}
 
\newcommand{\field}[1]{\mathbb{#1}}
\newcommand{\ZZ}{\ensuremath{{\field{Z}}}}
\newcommand{\CC}{\ensuremath{{\field{C}}}}
\newcommand{\RR}{\ensuremath{{\field{R}}}}
\newcommand{\PP}{\ensuremath{{\field{P}}}}

\newcommand{\commentout}[1]{}

\newcommand{\cA}{\ensuremath{{\mathcal{A}}}}

\newcommand{\cB}{\ensuremath{{\mathcal{B}}}}
\newcommand{\cC}{\ensuremath{{\mathcal{C}}}}

\newcommand{\cE}{\ensuremath{{\mathcal{E}}}}
\newcommand{\cF}{\ensuremath{{\mathcal{F}}}}

\newcommand{\cN}{\ensuremath{{\mathcal{N}}}}

\newcommand{\cR}{\ensuremath{{\mathcal{R}}}}
\newcommand{\cS}{\ensuremath{{\mathcal{S}}}}

\newcommand{\cT}{\ensuremath{{\mathcal{T}}}}

\newcommand{\cW}{\ensuremath{{\mathcal{W}}}}

\newcommand{\no}[1]{}

\title{Exceptional sequences of type $B_n/C_n$ and those in the abelian tube}
\author{Kiyoshi Igusa}
\email{igusa@brandeis.edu}

\author{Emre Sen}
\email{emresen641@gmail.com}



\begin{document}

\begin{abstract} We examine clusters in the cluster tube of rank $n+1$ using exceptional sequences in the abelian tube of rank $n+1$. Although the abelian tube has more exceptional sequences than the module categories of type $B_{n}/C_{n}$, we obtain a bijection between the set of signed exceptional sequences of any length in these categories. This bijection gives a reinterpretation of the formula of Buan-Marsh-Vatne comparing clusters of type $B_n/C_n$ with maximal rigid objects in the cluster tube of rank $n+1$. The bijection goes through the set of ``augmented'' rooted labeled trees.
\end{abstract}

\maketitle



\section*{Introduction}

This paper is about exceptional sequences in tubes. These are abelian tubes, not the cluster tubes although we start with the known result about cluster tubes.

{
We recall that an \emph{exceptional sequence} of \emph{length} $k$ is a sequence $(E_1,\cdots,E_k)$ of rigid bricks in $mod\text-\Lambda$ for some finite dimensional hereditary algebra $\Lambda$ so that $\Hom(E_j,E_i)=0=\Ext(E_j,E_i)$ for $i<j$. $E_i$ being a ``brick'' means all nonzero endomorphisms are isomorphisms. The exceptional sequence is \emph{complete} if $k$ is maximal, i.e., $k=n$ the rank of $\Lambda$ which is the number of simple modules, or the number of vertices in the quiver of $\Lambda$. Exceptional sequences are widely studied in representation theory and combinatorics because of several well-known combinatorial models for exceptional sequences. For example, complete exceptional sequences for an algebra of type $A_n$ are well-known to be in bijection with trees with $n$ edges and $n+1$ labeled vertices. So, there are $(n+1)^{n-1}$ of them. \cite{Ig-Sen}, \cite{GIMO}.

An \emph{abelian tube} of rank $n$ is a component of the Auslander-Reiten quiver of a tame hereditary algebra which is $\tau$-periodic: $\tau^n=id$. See Figure \ref{AR quiver W4} for an example. The tube is an infinite tube but only the ``mouth'' (bottom) of the tube is presented since higher objects are not bricks. Also the objects $V_{kk}$ are bricks but not rigid since $\tau V_{ij}=V_{i-1,j-1}$ and $\Hom(V_{kk},V_{k-1,k-1})\neq0$. For a hereditary algebra, $M$ is rigid if and only if $\Hom(M,\tau M)=0$, i.e., $M$ is ``$\tau$-rigid.'' One good thing about tubes is that they are independent of the ambient category. So, we may choose a convenient ambient category $mod\text-\Lambda$ where the tube is embedded.
}

By Buan, Marsh and Vatne \cite{BMV}, isomorphism classes of clusters in the cluster category of type $B_n$ or $C_n$ are in bijection with maximal rigid objects in the cluster tube of rank $n+1$. In this paper, we examine this result in terms of exceptional sequences. We compute the number of length $k$ exceptional sequences for $C_{n-1}$ and for the abelian tube $\cW_n$ of rank $n$. These numbers do not agree! However, the number of length $k$ signed exceptional sequence for $C_{n-1}$ and $\cW_n$ are equal. For example, when $n=4$, $C_3$ have $3^3=27$ complete exceptional sequences and $4\cdot 5\cdot 6=120$ complete signed exceptional sequences. The tube $\cW_4$ has $4^3=64$ complete exceptional sequences and $120$ complete signed exceptional sequences which are in bijection with the complete signed exceptional sequences of $C_3$ and this bijection holds for length $k$ signed exceptional sequences.

The bijection for signed exceptional sequences of any length is expected since, by an extension of the bijection of \cite{IT13}, ordered clusters in the cluster tube are in bijection with signed exceptional sequences in the abelian tube $\cW_n$. We use $\cT_n$ for the set of rooted labeled trees with $n$ vertices. We also usually stick to type $C_n$ and reserve $B_n$ to denote the braid group on $n$ strands.



We start with the ``obvious'' bijection between indecomposable modules of type $C_n$ and the bricks in the tube $\cW_n$ including the non-rigid bricks. This bijection sends exceptional sequences of type $C_n$ to ``soft'' exceptional sequences in $\cW_n$ where ``soft'' means we allow non-rigid bricks in the exceptional sequence (See Figure \ref{AR quiver W4}). There is also the notion of ``weak'' exceptional sequences \cite{Sen2} in cyclic Nakayama algebras of rank $n$. These are combinatorially equivalent to soft exceptional sequences in the tube $\cW_n$. But we use $\cW_n$ in order to associate this with the work of \cite{BMV}.

\begin{figure}[htbp]
\begin{center}
\begin{tikzpicture}[xscale=1.9,yscale=1.3]
\draw[dashed] (0.2,0)--(5.5,0);
\foreach \x/\y in {(1,0)/30,(2,0)/01,(3,0)/12,(4,0)/23,(5,0)/30}
\draw[fill,white] \x circle[radius=3mm];

\foreach \x/\y in {(1,0)/30,(2,0)/01,(3,0)/12,(4,0)/23,(5,0)/30}
\draw \x node{$V_{\y}$};
\foreach \x in {(1,0),(2.5,2.25),(3,0),(3.5,0.75)}
\draw[thick] \x circle[radius=2.5mm];
\foreach \x/\y in {(0.5,.75)/20,(1.5,.75)/31,(2.5,0.75)/02,(3.5,0.75)/13,(4.5,0.75)/20}
\draw \x node{$V_{\y}$};
\foreach \x/\y in {(0.5,2.25)/11,(1.5,2.25)/22,(2.5,2.25)/33,(3.5,2.25)/00,(4.5,2.25)/11}
\draw \x node{$V_{\y}$};
\foreach \x/\y in {(1,1.5)/21,(2,1.5)/32,(3,1.5)/03,(4,1.5)/10,(5,1.5)/21}
\draw \x node{$V_{\y}$};
\foreach \x in {0,...,5}\draw (\x.25,.4) node{$\nearrow$};
\foreach \x in {0,...,4}\draw (\x.75,1.1) node{$\nearrow$};
\foreach \x in {0,...,5}\draw (\x.25,1.9) node{$\nearrow$};
\foreach \x in {0,...,4}\draw (\x.75,.35) node{$\searrow$};
\foreach \x in {0,...,5}\draw (\x.25,1.1) node{$\searrow$};
\foreach \x in {0,...,4}\draw (\x.75,1.85) node{$\searrow$};
\end{tikzpicture}
\caption{The Auslander-Reiten quiver at the mouth of a tube of rank $4$. The objects $V_{kk}$ are bricks which are not rigid. $(V_{30},V_{33},V_{12},V_{13})$ is a soft exceptional sequence in $\cW_4$. $(V_{30},V_{12},V_{13})$ is an exceptional sequence.}
\label{AR quiver W4}
\end{center}
\end{figure}
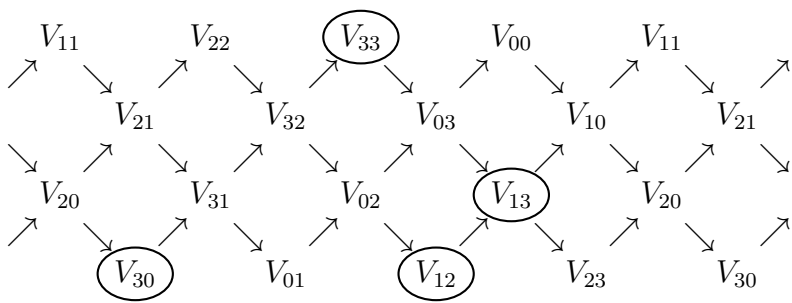

Every complete soft exceptional sequence for $\cW_n$ has exactly one nonrigid object $V_{kk}$. The positions in the Auslander-Reiten quiver determine a rooted labeled tree. We call value $\varepsilon=k$ the ``augmentation'' of the tree. This tells where the tree starts.

There is also an oriented chord diagram where object $V_{ij}$ is represented by a chord from $i$ to $j$ in a circle with $n$ marked points. The nonrigid object $V_{kk}$, which is the root of the tree, gives a loop at the augmentation point $k$ in the chord diagram. We call an oriented chord diagram with a loop a ``pointed'' chord diagram. The loop determines the orientation of the chords: clockwise around the loop. The correspondence is summarized as follows.

\begin{customthm}{A}\label{thm A}
For $k\le n$, there are natural bijections between the following sets having $n^k\binom nk$ elements.
\begin{enumerate}
	\item $\cE_n^{(k)}=$ the set of length $k$ exceptional sequences of $\Lambda$-modules for $\Lambda$ an hereditary algebra of type $B_n$ or $C_n$,
	\item $\cB_n^{(k)}=$ the set of length $k$ soft exceptional sequences for $\cW_n$, the abelian tube of rank $n$, where ``soft'' means the objects are not all rigid,
	\item $\widetilde\cC_n^{(k)}=$ the set of pointed chord diagrams with either $k$ oriented chords or $k-1$ oriented chords and one loop in a circle with $n$ marked points.
\end{enumerate}
Furthermore, for $k=n$, these sets are in bijection with $\widetilde\cT_n=$ the set of augmented rooted labeled trees with $n$ vertices.
\end{customthm}
Standard chord diagrams for $B_n/C_n$ are 2-periodic diagrams on circles with $2n$ marked points where our ``loops'' correspond to chords through the center of the circle. See \cite{Ath}, \cite{Reading} for related constructions. We prefer the smaller diagrams with $n$ marked points. Figure \ref{F2:W4 trees aug rtd} shows an example, for $n=4$. 
\begin{figure}[htbp]
\begin{center}
\begin{tikzpicture}[scale=.9]
\coordinate (T) at (-1.5,1);
\coordinate (E) at (2.5,1);
\coordinate (A) at (0,0.3);
\coordinate (B) at (1,1.7);
\coordinate (C) at (1.5,1);
\coordinate (D) at (1,0.3);
\begin{scope}
\draw (T) node{$(T,\varepsilon)=$};
\draw (E) node{,  $\varepsilon=3$};
\draw[thick] (A)--(B)--(C)--(D);
	\foreach \x/\y in {A/1,B/2,C/4,D/3}
	\draw[fill,white] (\x) circle[radius=2mm];
	\foreach \x/\y in {A/1,B/2,C/4,D/3}
	\draw[thick] (\x) circle[radius=2mm] node{$\y$};
\end{scope}
\begin{scope} [xshift=7cm,yshift=1cm,scale=.6] 
\draw[thick] (0,0) circle[radius=25mm];
\draw[very thick] (0,-2.5)--(2.5,0)--(-2.5,0)--(0,2.5);
\draw[ thick,->] (-2.5,0)--(-.5,2);
\draw[ thick,->] (-2.5,0)--(.5,0);
\draw[ thick,->] (2.5,0)--(.5,-2);
\draw (0,2.5) node[above]{$2$};
\draw (-2.5,0) node[left]{$1$};
\draw (2.5,0) node[right]{$3$};
\draw (0,-2.5) node[below]{$0$};
\draw (-1.5,1.5) node{\small $3$};
\draw (-.2,.3) node{\small $4$};
\draw[thick] (1.6,-.37) circle[radius=3mm] node{\small $2$};
\draw (1.3,-1.6) node{\small $1$};
\end{scope}
\end{tikzpicture}
\caption{Left is the augmented rooted labeled tree corresponding to the soft exceptional sequence in Figure \ref{AR quiver W4}. The root of the tree, vertex 2, with $\varepsilon=3$ corresponds to $V_{33}$, the second term of the sequence. At right is the corresponding pointed chord diagram. If $V_{ij}$ is the $k$th term in the soft exceptional sequence, arc $k$ will go from marked point $i$ to $j$.}
\label{F2:W4 trees aug rtd}
\end{center}
\end{figure}
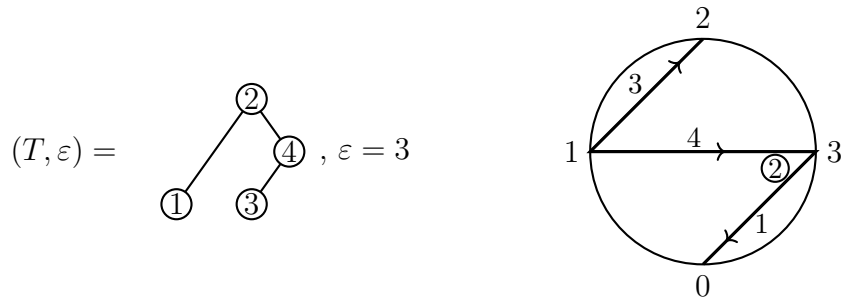

If we forget the root of the tree, but remember the augmentation of the root, we get an augmented forest with $n-1$ vertices. 
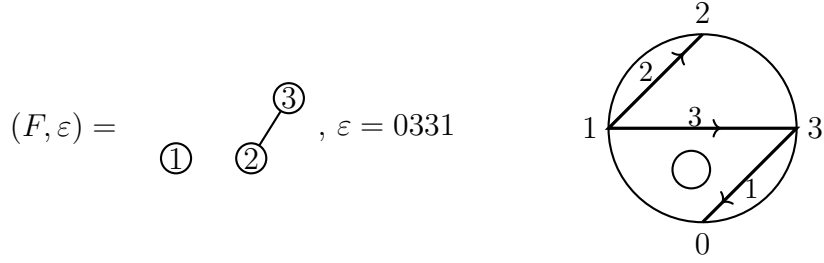
\begin{figure}[htbp]
\begin{center}
\begin{tikzpicture}
\coordinate (T) at (-1.5,.4);
\coordinate (E) at (2.8,.4);
\coordinate (A) at (0,0);
\coordinate (C) at (1.5,.8);
\coordinate (D) at (1,0);
\begin{scope}
\draw (T) node{$(F,\varepsilon)=$};
\draw (E) node{,  $\varepsilon=0331$};
\draw[thick] (A) (C)--(D);
	\foreach \x/\y in {A/1,C/3,D/2}
	\draw[fill,white] (\x) circle[radius=2mm];
	\foreach \x/\y in {A/1,C/3,D/2}
	\draw[thick] (\x) circle[radius=2mm] node{$\y$};
\end{scope}
\begin{scope} [xshift=7cm,yshift=4mm,scale=.5] 
\draw[thick] (0,0) circle[radius=25mm];
\draw[very thick] (0,-2.5)--(2.5,0)--(-2.5,0)--(0,2.5);
\draw[ thick,->] (-2.5,0)--(-.5,2);
\draw[ thick,->] (-2.5,0)--(.5,0);
\draw[ thick,->] (2.5,0)--(.5,-2);
\draw (0,2.5) node[above]{$2$};
\draw (-2.5,0) node[left]{$1$};
\draw (2.5,0) node[right]{$3$};
\draw (0,-2.5) node[below]{$0$};
\draw (-1.5,1.5) node{\small $2$};
\draw (-.2,.3) node{\small $3$};
\draw[thick] (-.3,-1.1) circle[radius=5mm] ;
\draw (1.3,-1.6) node{\small $1$};
\end{scope}
\end{tikzpicture}
\caption{Truncating the root in the augmented tree from Figure \ref{F2:W4 trees aug rtd} gives the  augmented forest on the left. The corresponding oriented chord diagram, on the right, has a unique central region where the missing loop used to be. All chords are oriented clockwise around it.}
\label{Fig: example of forest and directed chord diagram}
\end{center}
\end{figure}

There are $n$ possible ways to put back the root. In our example, $n=4$ and the augmentation sequence $0331$ assigns an augmentation for each possible root. Thus, there is an $n$-to-1 correspondence between augmented trees and augmented forests and there are $n^{n-1}$ augmented forests. Removing the root corresponds to removing the nonrigid object in a soft exceptional sequence to obtain a standard exceptional sequence. In the pointed chord diagram, this corresponds to removing the loop and producing a standard directed chord diagram. We obtain the following.

\begin{customthm}{B}\label{thm B}
For $k<n$ there is a natural bijection between the following $n^k\binom{n-1}k$ element sets.

$\cR_n^{(k)}=$ the set of length $k$ exceptional sequences in $\cW_n$,

$\cC_n^{(k)}=$ the set of length $k$ oriented chord diagrams on a circle with $n$ marked points.

\noindent Furthermore, for $k=n-1$ these sets are in bijection with $\widetilde\cF_{n-1}=$ the set of augmented forests with $n-1$ vertices.
\end{customthm}


Thus, $\cW_n$ has more exceptional sequences of length $k<n$ than $C_{n-1}$ (since $n^k>(n-1)^k$). However, they have the same number of signed exceptional sequences.

\begin{customthm}{C}\label{thm C}
There is a 1-1 correspondence between signed exceptional sequences in $\cW_n$ of any length $k<n$ and those in $mod\text-\Lambda$ of the same length for $\Lambda$ of type $C_{n-1}$.
\end{customthm}

This follows from Lemma \ref{lem D}. The size of these sets in \eqref{eq: number of length k signed exc seq} follows from Theorem \ref{thm E}.

\begin{customlem}{D}\label{lem D} Let $\cW$ be any wide subcategory of $mod\text-\Lambda$ without projective objects. Then
there is a bijection between signed exceptional sequences of length $k$ in $\cW$ and ordered $k$-tuples of ext-orthogonal rigid objects in $\cW$.
\end{customlem}

Theorem \ref{thm C} follows from Lemma \ref{lem D} and the results of \cite{BMV} and \cite{IT13} which, as we mentioned above, give a bijection between ordered $k$-tuples of ext-orthogonal rigid objects in $\cW_n$ and signed exceptional sequences of length $k$ for $C_{n-1}$. A detailed description of the bijection in Theorem \ref{thm C} is given in section \ref{ss: bijection signed exc seq Wn and Cn-1}.

We also compute the number of elements in this set using the probability distribution of relative projectives in exceptional sequences in $mod\text-\Lambda$.

\begin{customthm}{E}\label{thm E}
Let $(E_1,\cdots,E_k)$ be a random exceptional sequences of length $k$ for $C_n$. Then, the probability that $E_j$ is relatively projective is
\[
	\mathbb P(E_j\text{ is rel proj})={\frac{k+1-j}n}.
\]
Furthermore, these probabilities are independent.
\end{customthm}

Combining this with Theorem \ref{thm A}, we conclude that the number of signed exceptional sequences of length $k$ for $C_n$ is

\begin{equation}\label{eq: number of length k signed exc seq}
	n^k\binom{n}k\prod (1+\PP(E_j\text{ is rel proj}))=\frac{(n+k)!}{k!(n-k)!}.
\end{equation}

Although the following is redundant by Theorem \ref{thm C}, we also count the number of signed exceptional sequences for the tube $\cW_n$ in the next theorem. The key point is the description of the relatively projective objects.

\begin{customthm}{F}\label{thm F}
Let $V_\ast=(V_1,\cdots,V_{n-1})$ be a complete exceptional sequence in $\cW_n$ with corresponding augmented forest $(F,\varepsilon)$.

(a) $V_j$ is relatively projective in the exceptional sequence $V_\ast$ if and only if $v_j$ is a descending vertex in the corresponding forest $F$, i.e., $v_j$ is a child of $v_k$ for some $k>j$.

(b) The probability that $V_j$ is relatively projective in a random exceptional sequence is
\[
	\mathbb P(V_j\text{ is rel proj})={\frac{n-j-1}n}.
\]

(c) These probabilities are independent.
\end{customthm}

In the example in Figure \ref{Fig: example of forest and directed chord diagram}, $v_2$ is a descending vertex. So, $V_{12}$ is relatively projective in the exceptional sequence $(V_{30},V_{12},V_{13})$ and $V_{30},V_{13}$ are not. Both of the corresponding signed exceptional sequences for $C_3$ given by Theorem \ref{thm C} can be represented using Theorem \ref{thm A} by a pointed chord diagram as shown in Figure \ref{Fig: sin exc seq in C3}.
\begin{figure}[htbp]
\begin{center}
\begin{tikzpicture}[scale=.7]
\coordinate (A) at (-1.73,1);
\coordinate (B) at (1.73,1);
\coordinate (C) at (0,-2);
\coordinate (AC) at (-.58,-1);

\begin{scope} 
\draw[thick] (0,0) circle[radius=20mm];
\draw[very thick] (B)--(A)--(C);
\draw[ thick,->] (A)--(.5,1);
\draw[ thick,->] (A)--(AC);
\draw (A) node[left]{$1$};
\draw (B) node[right]{$2$};
\draw (C) node[below]{$0$};
\draw (0,1.3) node{\small $2$};
\draw (-.6,-.3) node{\small $3$};
\draw[ thick] (-.55,-1.6) circle[radius=2.4mm] node{\small $1$};
\end{scope}
\end{tikzpicture}
\caption{The signed exceptional sequence in $C_3$ corresponding to the example in Figure \ref{Fig: example of forest and directed chord diagram} is $(P_3[1],S_2[1],X_2)$ or $(P_3[1],S_2,X_2)$ depending on whether or not $V_{12}$ is shifted. Both of these are given by the pointed chord diagram above which comes from the oriented chord diagram in Figure \ref{Fig: example of forest and directed chord diagram} by collapsing the arc from 3 to 0.}
\label{Fig: sin exc seq in C3}
\end{center}
\end{figure}
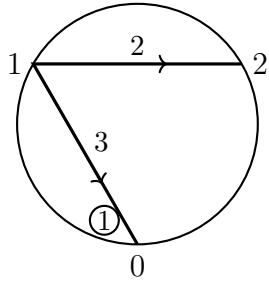

The fact that $\cW_n$ has no projective objects is reflected in the fact that the probability of the last object $V_{n-1}$ being relatively projective (and thus projective) is zero. Theorem \ref{thm F} implies that the number of complete signed exceptional sequences for $\cW_n$ is
\[
	n^{n-1}\prod (1+\PP(V_j\text{ is rel proj}))=\frac{(2n-2)!}{(n-1)!}
\]
which is the special case of \eqref{eq: number of length k signed exc seq} with $n,k$ both replaced with $n-1$.

Theorems \ref{thm E} and \ref{thm F} interpret the product formulas for the number of signed exceptional sequences and thus the number of clusters for $C_n$ or maximal rigid subsets of $\cW_n$ as a consequence of the product formula for the probability of independent events.

Finally, we discuss the action of the braid group $B_n$. By Ringel \cite{RingelExcSeq} (following {Crawley-Boevey} \cite{Crawley-Boevey} in the simply laced case), $B_n$ acts transitively on the set of complete exceptional sequence of type $C_n$. By Theorem \ref{thm A} we get an induced action of $B_n$ on the sets of augmented rooted labeled trees and pointed chord diagrams. We give a simple combinatorial description of these actions and deduce that the action of $B_n$ descends to an action on the smaller set of rooted labeled trees on $n$ vertices.

\section{Rooted labeled trees}\label{sec: rooted trees}

The relationship between trees, forests and exceptional sequences for quivers of type $A_{n-1}$ and $C_n$ is not new. It comes from \cite{Ringel2} and \cite{Ig-Sen}. In \cite{Ringel2} it is shown that complete exceptional sequences for hereditary algebras of type $C_n$ (or $B_n$) are given uniquely by a complete exceptional sequence of type $A_{n-1}$ and two integers modulo $n$. In \cite{Ig-Sen} a bijection is given between complete exceptional sequences for quivers of type $A_{n}$ and rooted labeled forests with $n$ vertices. This bijection was used to describe, in detail, the action of the braid group $B_n$ on this set of forests.

Thus, a complete exceptional sequence for an algebra of type $C_n$ is given by a rooted labeled forest with $n-1$ vertices together with two integers modulo $n$. In this paper, the first integer will be used to label a root and add it to the forest giving a rooted labeled tree with $n$ vertices. The second integer will added to the structure and be called an ``augmentation''. Thus, there will be $n^n$ augmented rooted labeled trees which, by \cite{Ringel2}, are in bijection with complete exceptional sequences for algebras of type $C_n$.

 The only new concept here is the action of $B_n$, the braid group on $n$ strands, on the set of trees and augmented trees. We also need to set up notation and terminology for the rest of the paper.

Let $\cT_n$ be the set of rooted trees with vertices labeled $1,\cdots,n$. Let $\cF_n$ be the set of rooted labeled forests with $n$ vertices. We sometimes consider $\cT_n$ as a subset of $\cF_n$. We also have a projection map
\[
    \pi:\cT_n\to \cF_{n-1}
\]
given by deleting the root of a tree, say $v_r$, and reducing by one the labels of vertices $v_i$ for $i>r$. Then $\pi$ is an $n$-to-$1$ map: for every $F\in \cF_{n-1}$ there are exactly $n$ elements in $\pi^{-1}(F)$ given by the possible labels of the root to be inserted. Since $\cF_{n-1}$ has $n^{n-2}$ elements, this implies that $\cT_n$ has $n^{n-1}$ elements.

\subsection{Braid group action on trees}\label{ss: braid action on trees}

Let $B_n$ be the braid group on $n$ strands. Recall that $B_n$ has $n-1$ generators $\sigma_1,\cdots,\sigma_{n-1}$ with the relations that $\sigma_i,\sigma_j$ commute if $|i-j|\ge2$ and
\begin{equation}\label{eq: braid relation}
\sigma_i\sigma_{i+1}\sigma_i=\sigma_{i+1}\sigma_i\sigma_{i+1}
\end{equation}
for $1\le i\le n-2$.

We recall from \cite{Ig-Sen} the action of $B_n$ on $\cF_n$, the set of rooted labeled forests. The action of a generator $\sigma_i\in B_n$ on a forest $F\in\cF_n$ depends on the relation between the vertices $v_i,v_{i+1}$ of $F$. There are several cases. Either one of these vertices is a child of the other, they are sibling, they are both roots, or they are not related in any of these ways. To reduce the number of cases, we add a root to the forest, label the root $v_0$ and denote by $F_+$ the resulting rooted tree. In $F_+$, $v_0$ is the parent of every root of $F$ and the roots of $F$ are sibling in $F_+$. Every vertex of $F$ has a unique parent in $F_+$. We write $a\to b$ if $a$ is the parent of $b$.
\begin{enumerate}
\item[(0)] When $v_i,v_{i+1}$ are not related, $F'=\sigma_iF$ is given by taking the same forest and switching the labels of $v_i$ and $v_{i+1}$. We write: $v_i'=v_{i+1}$ and $v_{i+1}'=v_i$.
\item If $v_i$ is a parent of $v_{i+1}$ and a child of $v_k$ (in $F_+$), i.e., $v_k\to v_i\to v_{i+1}$, then $F'=\sigma_iF$ is given by making $v_{i+1}$ the child of $v_k$ and the parent of $v_i$ and keeping all other relations the same. Thus $F'$ has new vertices $v_i',v_{i+1}'$ related by $v_k\to v_{i+1}'\to v_i'$ and, for $j\neq i,i+1$, $v_i'\to v_j$ if and only if $v_i\to v_j$ and $v_{i+1}'\to v_j$ if and only if $v_{i+1}\to v_j$.
\item If $v_i$ is a child of $v_{i+1}$: $v_k\to v_{i+1}\to v_i$ then $F'=\sigma_iF$ is given by making $v_i,v_{i+1}$ sibling with parent $v_k$, switching the children of $v_i$, $v_{i+1}$ and keeping all other relations the same. Thus, $v_k\to v_i'$, $v_k\to v_{i+1}'$ and, for $j\neq i,i+1$, $v_i'\to v_j$ if and only if $v_{i+1}\to v_j$ and $v_{i+1}'\to v_j$ if and only if $v_i\to v_j$.
\item $v_i,v_j$ are sibling (or they are both roots in $F$) with the same parent $v_k$ in $F_+$. In this case, $F'=\sigma_iF$ is given by, first having $v_i,v_{i+1}$ swap their children, then by making $v_{i+1}$ the child of $v_i$. Thus $v_k\to v_i'\to v_{i+1}'$.
\end{enumerate}

Note that $\sigma_i$ takes Case 1 to Case 2, Case 2 to Case 3 and Case 3 back to Case 1. The composition is the identity since $v_i$, $v_{i+1}$ swap children twice. Thus $\sigma_i^3=id$ in these cases. Case 0 goes to Case 0 and $\sigma_i^2=id$ that case.

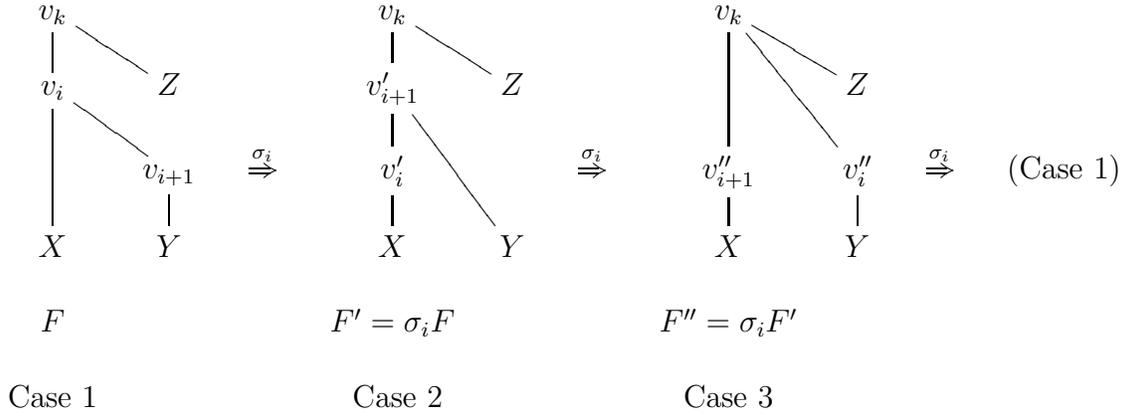
\begin{figure}[ht]
\[
\xymatrixrowsep{12pt}\xymatrixcolsep{10pt}
\xymatrix{
v_k\ar@{-}[d] &&&& v_k\ar@{-}[d] &  & && v_k\ar@{-}[d]\ar@{-}[dd]&  &  \\ 
v_i\ar@{-}[dd]\ar@{-}[dr]  & Z\ar@{-}[ul] &&& v_{i+1}'\ar@{-}[ddr] &Z\ar@{-}[ul]&&&  &  Z\ar@{-}[ul] &\\ 
 & v_{i+1}\ar@{-}[d] &   \ar@{=>}[r]^{\sigma_i}  & & v_i'\ar@{-}[d]\ar@{-}[u] & &  \ar@{=>}[r]^{\sigma_i}  && v_{i+1}''\ar@{-}[d] &  v_i''\ar@{-}[uul]\ar@{-}[d] & \ar@{=>}[r]^{\sigma_i} & & (\text{Case 1}) \\ 
	X&Y&&&X  &Y  & && X & Y\\
	F& &&&F'=\sigma_iF&   &&&F''=\sigma_iF' & && 
\\	\text{Case 1}  &&&& \text{ Case 2}  &&&&\text{Case 3} && 
	}
\]
\caption{Action of $\sigma_i$ on forests in Cases 1,2,3. $\sigma_i^3$ is the identity in these cases.}\label{fig: braid example}
\end{figure}

\begin{defn}\label{def: Bn action on Tn}
We define the action of the braid group $B_n$ on $\cT_n$ as follows. If $T\in \cT_n$ and $\sigma_i$ is a generator of $B_n$ we consider $T$ to be a forest and let $T'=\sigma_i(T)$ be given by the action of $\sigma_i$ on $\cF_n$ provided that $v_{i+1}$ is not the root of $T$. In these cases, $T'$ will be a tree. We have a \emph{special rule} when $v_{i+1}$ is the root of $T$: We define $T'=\sigma_i(T)$ to be the same tree with the labels of $v_i,v_{i+1}$ reversed. Thus, the root of $T'$ is its $i$th vertex $v_i'$ and $v_{i+1}'=v_i$.
\end{defn}

We will show that this defines an action of the braid group $B_n$ on the set $\cT_n$ using the fact that it is only slightly modified from the action of $B_n$ on $\cF_n$ from \cite{Ig-Sen} which was reviewed above. 

\begin{lem}\label{lem: action of Bn on Tn satisfies commutativity relations}
The action of $\sigma_i$ and $\sigma_j$ on $\cT_n$ commute if $|i-j|\ge2$.
\end{lem}

\begin{proof}
For $i,j\neq r-1$ the action of $\sigma_i,\sigma_j$ are given by their action on $\cF_n$. So, they commute by \cite{Ig-Sen}. For $i=r-1$, $\sigma_i$ is given by the special rule which only changes the labels of vertices $v_i,v_{i+1}$. This has no affect on the action of $\sigma_j$, so $\sigma_i,\sigma_j$ commute in this case as well.
\end{proof}

Recall that the \emph{fundamental braid} $\delta_n$ is the element of $B_n$ given by
\[\delta_n=\sigma_1\sigma_2\cdots \sigma_{n-1}\]

\begin{prop}\label{prop: delta T= T+}
    The action of the fundamental braid $\delta_n$ on any tree $T$ is given by simply increasing the labels of the vertices by 1 modulo $n$. In particular
    \[
    	\delta_n \sigma_i=\sigma_{i+1}\delta_n
    \]
    for $1\le i\le n-2$.
\end{prop} 

We use the notation $\delta_nT=T^+$ to indicate that the labels have been increased by $1$. An example is indicated in Figure \ref{Fig: example of delta3 on T3}.

Before the proof we show some consequences.

\begin{cor}\label{cor: braid relations hold}
Definition \ref{def: Bn action on Tn} gives an action of the braid group $B_n$ on $\cT_n$.
\end{cor}

\begin{proof}
By Lemma \ref{lem: action of Bn on Tn satisfies commutativity relations} it suffices to show the braid relation \eqref{eq: braid relation}. We do this first in the case $i=n-2$. In that case we have, by Proposition \ref{prop: delta T= T+}, that
\[
 \delta_{n-2}\sigma_{n-2}\sigma_{n-1}\sigma_{n-2}=\delta_n\sigma_{n-2}=\sigma_{n-1}\delta_n=\sigma_{n-1}\delta_{n-2}\sigma_{n-2}\sigma_{n-1}.
 \]
 But $\delta_{n-2}$ commutes with $\sigma_{n-1}$ since it is a product of $\sigma_j$ for $j<n-2$.  So
 \[
 	\sigma_{n-2}\sigma_{n-1}\sigma_{n-2}=\sigma_{n-1}\sigma_{n-2}\sigma_{n-1}.
 \]
 Conjugating this relation by $\delta_n^k$ for $k=n-2-i$ gives \eqref{eq: braid relation} for all other $i$.
\end{proof}

Recall \cite{Garside} that the \emph{Garside element} $\Delta\in B_n$, given by
\[
\Delta:=\delta_n\delta_{n-1}\cdots\delta_2,
\]
satisfies the well-known identity: $\Delta^2=\delta_n^n$. For example, for $n=3$ this says
\[
\Delta^2=(\delta_3\delta_2)^2=(\sigma_1\sigma_2\sigma_1)^2=(\sigma_1\sigma_2\sigma_1)(\sigma_2\sigma_1\sigma_2)=(\sigma_1\sigma_2)^3=\delta_3^3.
\]
 So, Proposition \ref{prop: delta T= T+} immediately implies the following.

\begin{cor}
$\Delta^2=\delta_n^n$ acts as the identity on $\cT_n$.\qed
\end{cor}

\begin{proof}[Proof of Proposition \ref{prop: delta T= T+}]
First, consider the case when the root of $T$ is the last vertex $v_n$.  Then, by the special rule, $\sigma_{n-1}$ will switch the labels of the last two vertices and $v_{n-1}'$ will be the new root of $T'=\sigma_{n-1}T$. Then $\sigma_{n-2}T'$ uses the special rule again moving the root to position $n-2$ and increases by $1$ the labels of $v_{n-2},v_{n-1}$. Proceeding in this way, $\delta_n$ will move the root to position $1$ and increase by $1$ the labels of all other vertices. This gives $T^+$.

If $v_n$ is not the root of $T$ then $v_{n-1}'$ will not be the root of $T'=\sigma_{n-1}T$. Proceeding in this way, $\delta_nT$ is computed without using the special rule. Thus the action of $\delta_n$ on $T$ is given by the action of $\delta_n$ on the exceptional sequence $E_\ast=(E_1,\cdots,E_n)$ corresponding to $T$ considered as a rooted labeled forest using \cite{Ig-Sen}. By \cite[Theorem 4.5]{Ig-Sen} we have
\[
	\delta_n(E_1,\cdots,E_n)=(\tau^\ast E_n,E_1,\cdots,E_{n-1})
\]
where $\tau^\ast E_n=\tau E_n$ when $E_n$ is not projective. Since the corresponding vertex $v_n$ is not the root of $T$, it must be an ascending vertex. So, by \cite{Ig-Sen}, $E_n$ is relatively injective and not relatively projective. In particular, $E_n$ is not projective. So, $\tau^\ast E_n=\tau E_n$.

If the support of $E_n$ is $(i,j]$ (making its weight $j-i$) then the support of $\tau E_n$ is $(i-1,j-1]$ with the same weight. If $E_i$ for $i<n$ has support inside the support of $E_n$, it cannot have $j$ in its support, otherwise there would be an epimorphism $E_n\to E_i$ which is not allowed. So, the support of $E_i$ is contained in $(i-1,j-1]$, the support of $\tau E_n$. So, $E_i<E_n$ implies $E_i<\tau E_n$ and conversely since $E_n$ and $\tau E_n$ have the same weight.

Similarly, if $E_n<E_i$ then the support of $E_i$ must be $(a,b]$ for some $a<i$. So, $\tau E_n<E_i$ and conversely since the weight of $E_i$ is fixed. Therefore, the rooted forest for $E_\ast$ and $\delta_nE_\ast$ are the same with labels shifted up by $1$ except for the last object which becomes the first. So, $\delta_nT=T^+$ in this case.
\end{proof}

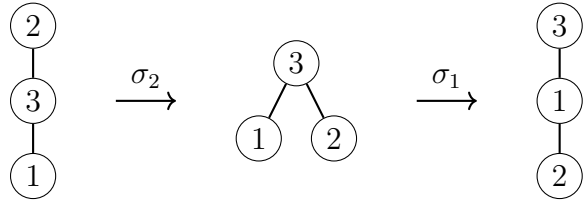
\begin{figure}[htbp]
\begin{center}
\begin{tikzpicture}
%
\begin{scope} 
\draw[thick] (0,1)--(0,-1);
\foreach \x/\y in {(0,0)/2,(0,1)/3,(0,-1)/1}
\draw[white,fill] \x circle[radius=3mm];
\foreach \x/\y in {(0,1)/2,(0,0)/3,(0,-1)/1}
\draw \x circle[radius=3mm] node{$\y$};
\end{scope}
\begin{scope}[xshift=-2cm]
\draw[thick,->] (3.1,0)--(3.9,0);
\draw (3.5,.3)node{$\sigma_2$};
\end{scope}
\begin{scope}[xshift=2.5cm]
\draw[thick] (0.5,-.5)--(1,.5)--(1.5,-.5);
\foreach \x/\y in {(1,.5)/3,(1.5,-.5)/2,(.5,-.5)/1}
\draw[white,fill] \x circle[radius=3mm];
\foreach \x/\y in {(1,.5)/3,(1.5,-.5)/2,(.5,-.5)/1}
\draw \x circle[radius=3mm] node{$\y$};
\end{scope}
\begin{scope}[xshift=2cm]
\draw[thick,->] (3.1,0)--(3.9,0);
\draw (3.5,.3)node{$\sigma_1$};
\end{scope}
\begin{scope}[xshift=7cm]
\draw[thick] (0,1)--(0,-1);
\foreach \x/\y in {(0,1)/3,(0,0)/1,(0,-1)/2}
\draw[white,fill] \x circle[radius=3mm];
\foreach \x/\y in {(0,1)/3,(0,0)/1,(0,-1)/2}
\draw \x circle[radius=3mm] node{$\y$};
\end{scope}
\end{tikzpicture}
\caption{$\delta_3=\sigma_1\sigma_2$ takes tree $T$ to the same tree with labels increased by 1 modulo $n=3$. In this example, the special rule is not invoked.}
\label{Fig: example of delta3 on T3}
\end{center}
\end{figure}

\subsection{Augmented trees and forests}\label{ss: augmented trees and forests}

We define an \emph{augmentation} of a tree $T$ to be an element of $\varepsilon\in\ZZ_n$ which we view as attached to the root of $T$ (independent of the label of the root). Thus, the set of augmented rooted labeled trees is
\[
\widetilde \cT_n=\cT_n\times \ZZ_n\cong \cF_{n-1}\times \{1,2,\cdots,n\}^2.
\]
Thus, $\widetilde\cT_n$ has $n^n$ elements which we already know, by \cite{Ringel2} and \cite{Ig-Sen}, are in bijection with complete exceptional sequences for algebras of type $C_n$. What may be more interesting is the set of augmented forests, defined below, which are in bijection with complete exceptional sequences in $\cW_n$, the tube of rank $n$.

Recall \cite{Ig-Sen} that the \emph{weight} of a vertex $v_i$ in a rooted forest $F$ is the number of vertices $\le v_i$ in $F$. We define the \emph{reduced (forest) weight} $\overline w(v_i)$ to be equal to the weight of $v_i$ if $v_i$ is a root of $F$, otherwise, $\overline w(v_i)=0$. Note that the sum of the weights of all the roots is equal to $n-1$, the number of vertices of $F$.

If $T\in\cT_n$ is a rooted labeled tree with root $v_r$ and $\pi(T)=F\in\cF_{n-1}$ is its underlying forest, recall first that the vertices of $F$ are $v_i^F=v_i^T$ if $i<r$ and $v_i^F=v_{i+1}^T$ if $i\ge r$. We define the \emph{reduced (tree) weight} $\overline w(v^T_i)$ or a vertex of $T$ to be the reduced forest weight of the corresponding vertex of $F$ if $i\neq r$ and $\overline w(v^T_r)=1$. Then the sum of reduced weights of the vertices of $T$ is equal to $n$, the number of vertices of $T$.

\begin{defn}\label{def: augmented forest}
An \emph{augmentation} of a rooted forest $F\in \cF_{n-1}$ is a function
\[
	\varepsilon:\{0,1,2,\cdots,n-1\}\to \mathbb Z_n
\]
with the property that
\[
	\varepsilon(i)=\varepsilon(i-1)-\overline w(v_i)
\]
for $0<i<n$. The pair $(F,\varepsilon)$ will be called an \emph{augmented (rooted) forest} and the set of all such pairs will be denoted $\widetilde\cF_{n-1}$.
\end{defn}
A forest augmentation $\varepsilon$ is uniquely determined by any of its values $\varepsilon(j)$ which can be chosen arbitrarily since
\[
	\varepsilon(k)=\varepsilon(j)-\sum_{j<i\le k}\overline w(v_i)
\]
for $k>j$ and similarly for $k<j$. In particular, $\varepsilon(n-1)=\varepsilon(0)+1$ (modulo $n$) since $\sum \overline w(v_i)=n-1$. Thus, $\widetilde\cF_{n-1}$ has $n^{n-1}$ elements.

There is an epimorphism $\widetilde\cT_n\to \widetilde\cF_{n-1}$ defined as follows. For any augmentation $\varepsilon_T$ of a tree $T$ with root $v_r$ we define the corresponding augmentation $\varepsilon_F$ of the underlying forest $\pi(T)=F$ to be the one determined by 
\begin{equation}\label{eq: def of epsilon F}
	\varepsilon_F(r-1)=\varepsilon_T.
\end{equation}
This implies:
\[
	\varepsilon_F(r-2)=\varepsilon_F(r-1)+\overline w(v_{r-1}^F)=\varepsilon_T +\overline w(v_{r-1}^T).
\]
Since $v_r$ is the root of $T$, $T'=\sigma_{r-1}T$ is given by the special rule which simply interchanges the labels of vertices $v_r$ and $v_{r-1}$. Thus $T,T'$ have the same underlying forest $F$, but $v^F_{r-1}=v^T_{r-1}=v^{T'}_r$. This implies that the augmented tree $(T',\varepsilon_T+\overline w(v_{r-1}^{T}))$ determines the same augmentation $\varepsilon_F$ on $F$ as does $(T,\varepsilon_T)$.  

\begin{defn}\label{def: action of Bn on (T,e)}
The action of the braid group $B_n$ on $\widetilde\cT_n$ can be defined as follows. If $T$ is a labeled tree with root $v_r$ and $1\le i<n$ then
\begin{enumerate}
    \item $\sigma_i(T,\varepsilon)=(\sigma_iT,\varepsilon)$ if $i\neq r-1$.
    \item $\sigma_{r-1}(T,\varepsilon)=(\sigma_{r-1}T,\varepsilon+\overline w(v_{r-1}^T))=(T',\varepsilon+\overline w(v_r^{T'})).$ 
\end{enumerate}
\end{defn}

As we shown above, $\sigma_{r-1}(T,\varepsilon)$ has the same underlying augmented forest as $(T,\varepsilon)$. Consequently, if $v_n$ is the root of $T$ and $\varepsilon$ is an augmentation for $T$ then $(T,\varepsilon)$, $\sigma_{n-1}(T,\varepsilon)$, $\sigma_{n-2}\sigma_{n-1}(T,\varepsilon),\cdots,\delta_n(T,\varepsilon)$ all have the same underlying augmented forest $(F,\varepsilon_F)$ with $\varepsilon_F$ given by \eqref{eq: def of epsilon F}. Since $\delta_nT=T^+$ has root $v_1$  and $\varepsilon_F(0)=\varepsilon-1$ this implies
\begin{equation}\label{eq: delta on (T,e) when vn is root}
	\delta_n(T,\varepsilon)=(T^+,\varepsilon-1)
\end{equation}
when $v_n$ is the root of $T$.

\begin{lem}\label{lem: delta on (T,e)}
If $v_r$ is the root of $T$ and $r<n$ then
\[
	\delta_n(T,\varepsilon)=(T^+,\varepsilon)
\]
and the root of $T^+$ has label $r+1$.
\end{lem}

\begin{proof}
By Proposition \ref{prop: delta T= T+} we have $\delta_n(T,\varepsilon)=(T^+,\varepsilon')$. Since $r<n$, the special rule is never applied. So, the augmentation never changes and $\varepsilon'=\varepsilon$.
\end{proof}

If we iterate this $n$ times, the root of $T$ will be the last vertex only once. So, we subtract $1$ from $\varepsilon$ only once and we get:
\begin{thm}\label{thm: action of delta n on tilde Tn}
$\delta_n^n(T,\varepsilon)=(T,\varepsilon-1)$. Therefore, $\Delta^2=\delta_n^n$ acts as the identity on $\cT_n$ and $\Delta^{2n}=(\delta_n^n)^{n}$ acts as the identity on $\widetilde\cT_n$.\qed
\end{thm}

We need two lemmas to show that the action of $\sigma_i$ on $\widetilde T_n$ satisfies the braid relations.

\begin{lem}\label{lem: delta action on tilde Tn}
The action of $\delta_n$ on $\widetilde \cT_n$ satisfies the relation
\[
	\delta_n\sigma_i=\sigma_{i+1}\delta_n
\]
for $i\le n-2$.
\end{lem}

\begin{proof}
Let $v_r$ be the root of $T$. Then there are two cases.

Case (a): $r=n$ and $i\le r-2=n-2$. 

By \eqref{eq: delta on (T,e) when vn is root} we have $\delta_n(T,\varepsilon)=(T^+,\varepsilon-1)$ with root $v_0$. So, by Lemma \ref{lem: delta on (T,e)},
\[
	\sigma_{i+1}\delta_n(T,\varepsilon)=(\sigma_{i+1}(T^+),\varepsilon-1)=((\sigma_iT)^+,\varepsilon-1)=\delta_n(\sigma_iT,\varepsilon)=\delta_n\sigma_i(T,\varepsilon)
\]

Case (b): $r<n$ and $i\le n-2$.

If $i\neq r-1$ then $\varepsilon$ does not change and we have:
\[
	\delta_n\sigma_i(T,\varepsilon)=(\delta_n\sigma_iT,\varepsilon)=(\sigma_{i+1}\delta_n T,\varepsilon)=\sigma_{i+1}\delta_n (T,\varepsilon).
\]
When $i=r-1$, $\varepsilon$ changes to $\varepsilon+ \overline w(v_i^T)$. But $\overline w(v_i^T)=\overline w(v_{i+1}^{T^+})$ since $T,T^+$ have the same underlying forest. So,
\[
	\delta_n\sigma_i(T,\varepsilon)= \delta_n(\sigma_iT,\varepsilon+ \overline w(v_i^T))=(\delta_n\sigma_iT,\varepsilon+ \overline w(v_i^T))\qquad\qquad
\]
\[
	\qquad\qquad=(\sigma_{i+1}T^+,\varepsilon+\overline w(v_{i+1}^{T^+}))=\sigma_{i+1}(T^+,\varepsilon)=\sigma_{i+1}\delta_n(T,\varepsilon).
\]
This proves the lemma in all cases.
\end{proof}

\begin{lem}\label{lem: reduced weight identity}
Let $T\in \cT_n$ with root $v_r$ and suppose $i\neq r,r-1$. Let $v_j$ and $v_j'$ be the $i$th vertices of $T$ and $T'=\sigma_iT$ then
\[
	\overline w(v_i)+\overline w(v_{i+1})=\overline w(v_i')+\overline w(v_{i+1}')
\]
and $\overline w(v_j)=\overline w(v_j')$ for $j\neq i,i+1$.
\end{lem}

\begin{proof} The reduced weight of the root $v_r$ is equal to $1$ by definition. The other reduced weights add up to $n-1$. So, it suffices to show that $\overline w(v_j)=\overline w(v_j')$ for $j\neq i,i+1,r$.
Let $E_\ast=(E_1,\cdots,E_{n-1})$ and $E_\ast'=(E_1',\cdots,E_{n-1}')$ be the complete exceptional sequences for $A_{n-1}$ with straight orientation corresponding to the forests $F=\pi(T)$ and $F'=\pi(T')$. Let $v_k^F$ be the vertex corresponding to $v_j$. (Thus $k=j$ if $j<r$ and $k=j-1$ otherwise.) The braid move $\sigma_i$ does not change the perpendicular categories in which $E_k,E_k'$ lie. So, $E_k$ is relatively projective or injective in $E_\ast$ if and only if $E_k'$ is relatively projective or injective in $E_\ast'$. Therefore $v_k$ is a root of $F$ if and only if $v_k'$ is a root of $F'$. Also $E_k'=E_k$. So, the weight is unchanged. So, the reduced weight is unchanged.
\end{proof}

\begin{thm}\label{thm: action of Bn on (T,e)}
Definition \ref{def: action of Bn on (T,e)} gives an action of the braid group $B_n$ on $\widetilde \cT_n$.
\end{thm}

\begin{proof} This follows from Lemma \ref{lem: delta action on tilde Tn} in the same way that Corollary \ref{cor: braid relations hold} follows from Proposition \ref{prop: delta T= T+}. We just need to show the commutativity relation: that $\sigma_i,\sigma_j$ commute if $|i-j|\ge 2$. If $i,j\neq r-1$ this follows from Lemma \ref{lem: action of Bn on Tn satisfies commutativity relations} since $\sigma_i,\sigma_j$ do not change the augmentation $\varepsilon$.

Suppose $j=r-1$. Then $\sigma_j(T,\varepsilon)=(\sigma_jT, \varepsilon+\overline w(v_j^T))$. Since $|i-j|\ge2$ we have either $i\ge r+1$ or $i\le r-3$. By Lemma \ref{lem: reduced weight identity} we have $\overline w(v_j^T)=\overline w(v_j^{T'})$ for $T'=\sigma_iT$. Thus,
\[
	\sigma_j\sigma_i(T,\varepsilon)=\sigma_j(T',\varepsilon)=(\sigma_jT',\varepsilon+\overline w(v_j^{T'}))=\sigma_i(\sigma_jT,\varepsilon+\overline w(v_j^{T}))=\sigma_i\sigma_j(T,\varepsilon)
\]
since $\sigma_jT'=\sigma_j\sigma_iT=\sigma_i\sigma_jT$. So, $\sigma_i,\sigma_j$ commute and the proof is complete.
\end{proof}

\section{Oriented and pointed chord diagrams}\label{sec: chord diagrams}

Oriented chord diagrams are defined to be sets of oriented chords on a circle with $n$ marked points which are pairwise ``noncrossing'' and ordered in a good way. We also define ``pointed'' chord diagrams which are oriented chord diagrams together with an added loop at some vertex in the ``central region'' of the interior of the circle. We show that there is a bijection between oriented chords $\gamma_{ij}$ and objects $V_{ij}$ in $\cW_n$, the abelian tube of rank $n$.

\subsection{Unoriented strand and chord diagrams}\label{ss: unoriented chords}

We review the definition and enumeration of unoriented chord diagrams and strand diagrams from \cite{GIMO}.

Take a circle with $n+1$ marked points labeled $0$ to $n$. We usually order these clockwise around the circle, however, other orderings are sometimes convenient. (See Figure \ref{Fig: strand to chord diagram}.) By a \emph{chord} we mean a straight line segment connecting two of these marked points. There are $\binom{n+1}2=n(n+1)/2$ chords. A pair of chords is \emph{noncrossing} if they do not cross, i.e., do not meet except possibly at their endpoints. 

We define a \emph{chord diagram} to be a sequence $(c_1,\cdots,c_k)$ of pairwise noncrossing chords having the property that, at every marked point, the chords $c_i$ which have one endpoint at that marked point are ordered clockwise. It is easy to see that a chord diagram can have no oriented cycles since, in any oriented cycle, the chords must be numbered in increasing order as we go counterclockwise around the cycle and this is impossible. Thus, every region of the complement of a chord diagram has at least one part of its boundary being an arc on the circle.

\begin{lem}\label{lem: chord diagrams can be completed}
Every chord diagram of length $k<n$ on a circle with $n+1$ marked points can be completed to a chord diagram with $n$ chords.
\end{lem}

\begin{proof}
The circle is divided into $n+1$ arcs by the marked points. If $k<n$ then one of the $k+1$ regions in the complement of a chord diagram $(c_1,\cdots,c_k)$, call it $U$, has at least two arcs on its boundary. Suppose two arcs on the boundary of $U$ are consecutive with endpoints the marked points $m_1,m_2,m_3$ in clockwise order on the circle. Then $m_2$ is not an endpoint of any of the given chords $c_i$ and we take the chord connecting $m_2,m_3$ to the chord diagram and call it $c_{k+1}$. It will be clockwise from any $c_i$ incident to $m+3$. If the boundary of $U$ does not have consecutive arcs, we let $c_{k+1}$ be the chord connecting the rightmost marked points of two of the arcs. This will be clockwise from any $c_i$ incident to either of these marked points. Furthermore, $c_{k+1}$ will not cross any $c_i$ since $U$ is convex.
\end{proof}

\begin{thm}\label{thm: chord diagrams and exc seq of type An}
Chord diagrams of length $k$ on a circle with $n+1$ marked points are in bijection with exceptional sequences of length $k$ for any quiver of type $A_n$.
\end{thm}

The proof using strand diagrams is given below. The cardinality of this set is known. See, for example, \cite{I:prob}:

\begin{thm}\label{thm: number of exc seq for An-1}
The number of exceptional sequences of length $k$ for a quiver of type $A_n$ with any orientation is
\[
	\binom{n+1}{k+1}(n+1)^{k-1}.
\]
\end{thm}

We need this formula with $n$ reduced by 1:

\begin{cor}\label{cor: number of chord diagrams of length k}
The number of length $k$ chord diagrams on a circle with $n$ marked points is
\[
	\binom{n}{k+1}n^{k-1}=\frac1{k+1}\binom{n-1}k n^k.
\]
\end{cor}

In \cite{GIMO} it is shown that complete exceptional sequences for a quiver of type $A_n$ are in bijection with strand diagrams which are also in bijection with chord diagrams. We use this result, but with a different orientation convention. Namely, in our diagrams, the strands and chords are required to be ordered clockwise at each marked point.

In a quiver of type $A_n$, there are $n-1$ arrows which we number $\alpha_1,\cdots,\alpha_{n-1}$ where $\alpha_i$ connects vertex $i$ to vertex $i+1$. If $\alpha_i$ points from $i+1$ to $i$ we will call it a \emph{positive} arrow, otherwise, it is a \emph{negative} arrow. (This is the opposite of the sign convention in \cite{GIMO}.) We associate to each arrow $\alpha_i$ the point $(i,0)$ on the $x$-axis in the plane. We also take the points $(0,0)$ and $(n,0)$. For $0\le i<j\le n$, the \emph{strand} $c(i,j)$ is defined to be an isotopy class of arcs going from $(i,0)$ to $(j,0)$ in the $xy$-plane which is disjoint from the other points $(k,0)$ for integer $k$ and so that the arc has no vertical tangents except at endpoints and does not pass over any point $(k,0)$ if $\alpha_k$ is a positive arrow and does not pass under $(k,0)$ if $\alpha_k$ is a negative arrow.

To transform strands into chords, we change the second coordinate of each point $(i,0)$ so that the point lies on the circle with diameter $(0,0)$ to $(n,0)$ and so that the sign of the second coordinate is the sign of the arrow $\alpha_i$. Strands can then be straightened to form noncrossing chords. The indecomposable representations of the $A_n$ quiver are denoted $M_{ij}$ if the support of the representation is the half-open interval $(i,j]$. For example, $M_{12}$ is the simple module $S_2$. The module $M_{ij}$ is represented by the strand $c(i,j)$ and by the chord connecting marked points $i$ and $j$ on the circle. The theorem in \cite{GIMO} is that complete exceptional sequences are in bijection with complete chord diagrams (having $n$ chords) which are equivalent to complete strand diagrams where the ordering of the chords and strands is clockwise at each point. See also \cite{BGMS} and \cite[Fig 4]{AI}.

An example is given on the left side of Figure \ref{Fig: strand to chord diagram} for the indicated $A_5$ quiver.
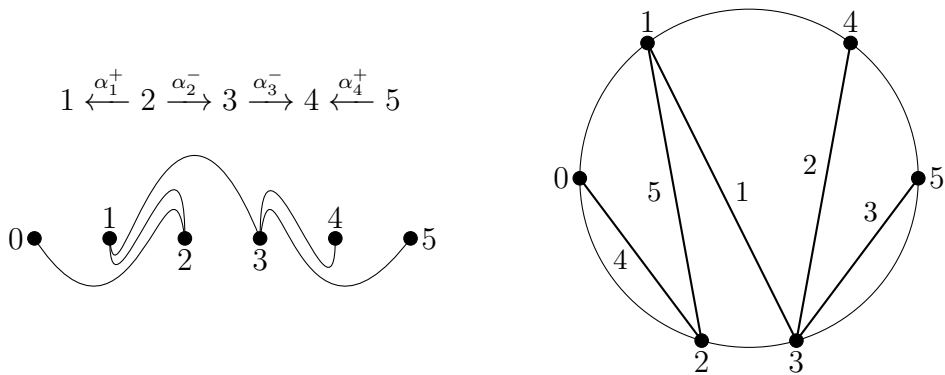
\begin{figure}[htbp]
\begin{center}
\begin{tikzpicture}
\clip (-.5,-2.65) rectangle (13,2.3);
\begin{scope}[yshift=-8mm]
\draw (2.6,2) node{$1\xleftarrow{\alpha_1^+} 2 \xrightarrow{\alpha_2^-} 3 \xrightarrow{\alpha_3^-} 4\xleftarrow{\alpha_4^+} 5$};
\foreach \x in {0,...,5} \draw[fill] (\x,0) circle[radius=.9mm];
\draw (0,0) node[left]{0};
\draw (1,0) node[above]{1};
\draw (2,0) node[below]{2};
\draw (3,0) node[below]{3};
\draw (4,0) node[above]{4};
\draw (5,0) node[right]{5};
\draw (0,0).. controls (1.3,-2) and (1.8,1.5)..(2,0);
\draw (1,0)..controls (1,-1.4) and (2,2)..(2,0);
\draw (1,0)..controls (1.1,-1.2) and (1.8,3)..(3,0);
\draw (3,0)..controls (3,2) and (4,-1.5)..(4,0);
\draw (3,0)..controls (3.2,1.5) and (3.5,-2)..(5,0);
\end{scope}
\begin{scope}[xshift=9.5cm,scale=.8]
\draw (0,0) circle[radius=25mm];
\coordinate (A0) at (-2.5,0);
\coordinate (A5) at (2.5,0);
\coordinate (A1) at (-1.5,2);
\coordinate (A4) at (1.5,2);
\coordinate (A2) at (-.7,-2.4);
\coordinate (A3) at (.7,-2.4);
\coordinate (C1) at (-.1,-.2);
\coordinate (C2) at (.9,.2);
\coordinate (C3) at (1.8,-.5);
\coordinate (C4) at (-1.9,-1.2);
\coordinate (C5) at (-1.4,-.2);
\foreach \x in {1,...,5}
\draw (C\x) node{\small$\x$};
\foreach \x in {(-2.5,0),(2.5,0),(-1.5,2),(1.5,2),(-.7,-2.4),(.7,-2.4)}
\draw[fill]  \x circle[radius=1mm];
\draw[thick] (A0)--(A2)--(A1)--(A3)--(A4) (A3)--(A5);
\draw (A0) node[left]{0};
\draw (A5) node[right]{5};
\draw (A2) node[below]{2};
\draw (A3) node[below]{3};
\draw (A1) node[above]{1};
\draw (A4) node[above]{4};
\end{scope}
\end{tikzpicture}
\caption{The exceptional sequence $(M_{13},M_{34},M_{35}, M_{02},M_{12})$ for the quiver of type $A_5$ at upper left corresponds to the strand diagram $(c(1,3),c(3,4),c(3,5), c(0,2),c(1,2))$ below it and the chord diagram on the right. Chords and strands are ordered clockwise at each vertex.
}
\label{Fig: strand to chord diagram}
\end{center}
\end{figure}

\begin{proof}[Proof of Theorem \ref{thm: chord diagrams and exc seq of type An}]
In \cite{GIMO} it is shown that complete exceptional sequences for a quiver of type $A_n$ are in bijection with strand and chord diagrams with $n$ strands/chords. Define equivalence relations on these sets: Two exceptional sequences are equivalent if their first $k$ terms are the same. Similarly, two chord diagrams are equivalent if the first $k$ chords are the same. By Lemma \ref{lem: chord diagrams can be completed}, any chord diagram with $k$ chords can be completed to a chord diagram with $n$ chords. Thus, the set of equivalence classes of complete chord diagrams are in bijection with chord diagrams of length $k$. Similarly, it is well-known that exceptional sequences of length $k$ on an hereditary algebra can be completed to an exceptional sequence of length $n$. So, length $k$ exceptional sequences are in bijection with equivalence classes of complete exceptional sequences. The bijection between complete exceptional sequences and complete chord diagrams induces a bijection between these equivalence classes proving the theorem.

(We could also quote the proof of the theorem of \cite{GIMO} which implies our theorem.)
\end{proof}

\subsection{Oriented chord diagrams}\label{ss: oriented chord diagrams}

Take a circle with $n$ marked points labeled $0$ to $n-1$ in clockwise order. For distinct $i,j$ in this set, the \emph{oriented chord} $\gamma_{ij}$ is defined to be the oriented line segment going from marked point $i$ to marked point $j$ on the circle. Let $\cC_n$ denote the set of such chords. Then $\cC_n$ has $n(n-1)$ elements. We define the compatibility relation on this set.

Each chord $\gamma_{ij}$ divides the interior of the circle into two convex open regions which we call its ``left side'' and ``right side''. The \emph{left side} is bounded by $\gamma_{ij}$ and the arc on the circle which goes clockwise from marked point $i$ to marked point $j$. The \emph{right side} of the oriented chord is bounded by $\gamma_{ij}$ and the arc on the circle going counterclockwise from $i$ to $j$. We sometimes call the left side of $\gamma_{ij}$ its \emph{support}.

We say that two oriented chords are \emph{compatible} or \emph{noncrossing} if, either their supports are disjoint (but not complementary), or the support of one is contained in the support of the other. By ``not complementary'' we mean that $\gamma_{ij},\gamma_{ji}$ are not compatible. It is clear that compatible chords do not intersect in their interiors. By a \emph{compatible pair} we mean an ordered pair of compatible oriented chords having the property that either they are disjoint, or they meet at a common endpoint, say $k$, and the second chord is clockwise from the first at endpoint $k$. See Figure \ref{fig: compatible chords}.

\begin{figure}[htbp]
\begin{center}
\begin{tikzpicture}[scale=.7]
\coordinate (A0) at (0,-2);
\coordinate (A1) at (-1.73,-1);
\coordinate (A2) at (-1.73,1);
\coordinate (A3) at (0,2);
\coordinate (A13) at (-.87,.5);
\coordinate (A113) at (-1.1,.5);
\coordinate (A443) at (.9,-.5);
\coordinate (A43) at (.87,-.5);
\coordinate (A4) at (1.73,1);
\coordinate (A5) at (1.73,-1);

\coordinate (X) at (-.4,-1.1);

\draw (X) node{$X$};
\draw (A1) node[left]{$k$};
\draw[thick] (0,0) circle[radius=20mm];
\foreach \x in {0,...,5}
\draw[fill] (A\x) circle[radius=1mm];
\draw[very thick,->] (A3)--(A13) (A1)--(A13);
\draw[very thick,->] (A4)--(A43) (A0)--(A43);
\draw[very thick,->] (0,0)--(A1) (A4)--(0,0) ;
\draw (A113) node[above]{$\alpha$};
\draw (A443) node[right]{$\gamma$};
\draw (0,.1) node[above]{$\beta$};
\end{tikzpicture}
\caption{$(\alpha,\beta)$ is a compatible pair since $\beta$ is clockwise from $\alpha$ at their common endpoint $k$ and the left side of $\alpha$ is disjoint from the left side of $\beta$. $\beta,\gamma$ are not compatible since region $X$ is on the left side of both. $\alpha,\gamma$ are compatible since the left side of $\alpha$ is contained in the left side of $\gamma$. Since $\alpha,\gamma$ do not share an endpoint, they are compatible in either order.}
\label{fig: compatible chords}
\end{center}
\end{figure}
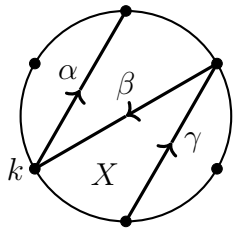

\begin{defn}
An \emph{ordered chord diagram} on a circle with $n$ marked points is defined to be an ordered sequence of oriented chords $(\gamma_1,\cdots,\gamma_k)$ having the property that, for any $1\le i<j\le k$, $(\gamma_i,\gamma_j)$ is a noncrossing pair of oriented chords and with the further property that the chords form a forest (disjoint union of trees). Equivalently, the set of underlying unoriented chords does not contain a cycle. In particular this means $k<n$, 
\end{defn}

In Figure \ref{fig: compatible chords}, $(\alpha,\beta), (\alpha,\gamma)$ and $(\gamma,\alpha)$ are oriented chord diagrams with 2 chords whereas $(\beta,\alpha)$ and $(\beta,\gamma)$ are not. Also, any oriented chord is an oriented chord diagram by itself.

\begin{defn}\label{def: central region}
Any oriented chord diagram $(\gamma_1,\cdots,\gamma_k)$ divides the interior of the circle into $k+1$ convex open regions. We define the \emph{central region} to be the unique region on the right side of every chord $\gamma_i$. 
\end{defn}

It is easy to see, by induction on $k\ge0$, that the central region exists and is unique. Indeed, if $U_{k-1}$ denotes the central region of the complement of $\gamma_i$ for $i<k$, either $\gamma_k$ is on the left side of some $\gamma_i$, placing $U_{k-1}$ on its right side, or $\gamma_k$ is on the right side of every $\gamma_i$ before it in which case $\gamma_k$ cuts $U_{k-1}$ in two parts, one of which is the central region $U_k$ of the complement of all the chords.

By definition of left and right sides it follows that all chords in an oriented chord diagram go clockwise around the central region. Thus, the orientation of each chord is determined by the location of the central region.

More precisely, we have the following.
\begin{lem}\label{lem: central region determines orientation}
If we forget the orientation, any oriented chord diagram gives an unoriented chord diagram of the same length. Conversely, given any unoriented chord diagram of length $k$, any of the $k+1$ regions in its complement can be designated as the central region to construct an oriented chord diagram of the same length.
\end{lem}
This allows us to count the number of oriented chord diagrams.

\begin{thm}\label{thm: cardinality of Cn(k)}
The number of oriented chord diagrams of length $k$ in a circle with $n$ marked points is
\[
	\binom{n-1}k n^k
\]
\end{thm}

\begin{proof}
By the lemma, every unoriented chord diagram of length $k$ gives $k+1$ oriented chord diagrams of the same length. So, we multiply the number in Corollary \ref{cor: number of chord diagrams of length k} by $k+1$ to obtain the theorem.
\end{proof}

We denote by $\cC_n^{(k)}$ the set of oriented chord diagrams of length $k$ in a circle with $n$ marked points numbered clockwise from $0$ to $n-1$. Theorem \ref{thm: cardinality of Cn(k)} gives the cardinality of this set. We call an oriented chord diagram \emph{complete} if it has the maximum number of chords which is $k=n-1$. By Theorem \ref{thm: cardinality of Cn(k)}, the number of complete oriented chord diagrams is $n^{n-1}$. We will show that these are in bijection with augmented forests with $n-1$ vertices. But first, we need to add loops to complete oriented chord diagrams to form ``pointed chord diagrams'' which will be in bijection with the $n^n$ augmented rooted trees with $n$ vertices.

\subsection{Loops and pointed chord diagrams}\label{ss: pointed chord diagrams}
We add to oriented chord diagrams objects which we call \emph{loops} and denote by $\gamma_{kk}$. This will be a formally defined objects which we draw as a circle at marked point $k$ which is required to be disjoint from all chords in the chord diagram. The set of possible oriented chords and loops, denoted $\widetilde \cC_n$, is defined to be the set of all $\gamma_{ij}$ where $i,j$ are not necessarily distinct integers modulo $n$. Thus $\widetilde \cC_n$ has $n^2$ elements.

A sequence of $k$ elements $(\gamma_1,\cdots,\gamma_k)$ in the set $\widetilde \cC_n$ will be called \emph{noncrossing} if at most one of the $\gamma_i$ is a loop, the others form a sequence of noncrossing oriented chords. Furthermore, the loop, say $\gamma_\ell$, if one exists in the sequence, must lie in the central region if it is places at marked point $\ell$ in the position which is clockwise from all chords $\gamma_i$ at $\ell$ for $i<\ell$ and counterclockwise from all chords $\gamma_j$ at $\ell$ for $j>\ell$. An example is drawn in Figure \ref{fig: example of chord diagram with loop} where we observe that the orientation of all chords are determined by the central region which is determined by the loop and the position of the loop at marked point $\ell$ is determined by its position ($5$th) in the sequence.

Let $\widetilde\cC_n^{(k)}$ denote the set of noncrossing sequences of $k$ elements of $\widetilde\cC_n$. Elements of $\widetilde\cC_n^{(n)}$, the maximum size, will be called \emph{pointed chord diagrams}.

\begin{thm}\label{thm: number of pointed chord diagrams}
    The set $\widetilde\cC_n^{(n)}$ has $n^n$ elements.
\end{thm}

\begin{proof}
    By Corollary \ref{cor: number of chord diagrams of length k}, there are $n^{n-2}$ unoriented chord diagrams of maximal length $k=n-1$. A loop can be added to this at any of the $n$ positions in the sequence and at any of the $n$ marked points. Thus, there are $n^2$ choices for the loop. The loop will determine the central region and the thus the orientation of all the chords, giving a pointed chord diagram. Thus, there are $n^{n-2}\cdot n^2=n^n$ pointed chord diagrams.
\end{proof}

This leads us to Theorem \ref{thm A}. In fact, the proof of the theorem above and the definition of augmented rooted trees mirrors the proof in \cite{Ringel2} that there are $n^n$ complete exceptional sequences for an algebra of type $B_n$ or $C_n$.

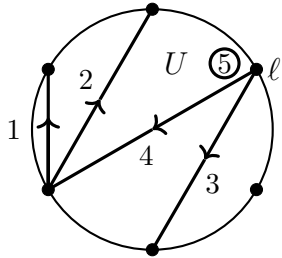
\begin{figure}[htbp]
\begin{center}
\begin{tikzpicture}[scale=.8]
\coordinate (A0) at (0,-2);
\coordinate (A1) at (-1.73,-1);
\coordinate (A2) at (-1.73,1);
\coordinate (A3) at (0,2);
\coordinate (A13) at (-.87,.5);
\coordinate (A113) at (-1.1,.5);
\coordinate (A443) at (.7,-.9);
\coordinate (A43) at (.87,-.5);
\coordinate (A4) at (1.73,1);
\coordinate (A5) at (1.73,-1);

\coordinate (X) at (.4,1.1);

\draw (X) node{$U$};
\draw (A4) node[right]{$\ell$};
\draw[thick] (0,0) circle[radius=20mm];
\foreach \x in {0,...,5}
\draw[fill] (A\x) circle[radius=1mm];
\draw[very thick,->] (A3)--(A13) (A1)--(A13);
\draw[very thick,->](A0)--(A43) (A4)--(A43) ;
\draw[very thick,->] (0,0)--(A1) (A4)--(0,0) ;
\draw[very thick,->] (A2)--(A1) (A1)--(-1.73,0.2) ;
\draw (A113) node[above]{\small$2$};
\draw (A443) node[right]{\small$3$};
\draw (-.1,-.1) node[below]{\small$4$};
\draw (-2,0) node[left]{\small$1$};
\draw (1.2,1.1) node{\small$5$};
\draw[very thick] (1.2,1.1) circle[radius=2.4mm];
\end{tikzpicture}
\caption{A noncrossing sequence of four oriented chords and one loop. The chords are the first four in the sequence and the loop, at $\ell$, is the $5$th term in the sequence. The loop is required to be clockwise from chords 3 and 4. It is also required to lie in the central region, denoted $U$.}
\label{fig: example of chord diagram with loop}
\end{center}
\end{figure}

\section{Theorem \ref{thm A}}\label{sec: Theorem A}

We prove Theorem \ref{thm A} which may be obvious to many readers. The purpose is to set up the notation which will make the bijections very clear. The notation $M_{ij}$, used in section \ref{ss: unoriented chords} for representation of quivers of type $A_{n-1}$, will be extended to (isomorphism classes of) indecomposable modules over an hereditary algebra of type $C_n$. There will be corresponding objects $V_{ij}$, which are bricks, in the tube $\cW_n$ of rank $n$, as outlined in the introduction. These will correspond to the objects $\gamma_{ij}$ (oriented chords or loops) from section \ref{ss: pointed chord diagrams} above. The indicated bijection between these three sets of objects will induce bijections between corresponding sets of exceptional sequences. The results of \cite{Ig-Sen} will be used to obtain the bijection with augmented trees. We also use \cite{IM} to compute the number of exceptional sets in $mod\text-\Lambda$ and in $\cW_n$.

\subsection{Exceptional sequences in $mod\text-\Lambda$}\label{ss: exc seq in mod Lambda}

Let $\Lambda$ be an hereditary algebra of type $C_n$ with straight orientation of the modulated quiver with the long root at the last vertex which is a source. This is sometimes written as:
\[
	K\leftarrow K\leftarrow\cdots \leftarrow K\leftarrow F
\]
where $K,F$ are fields and $F$ is a degree 2 extension of $K$. Thus, the first $n-1$ vertices form a quiver of type $A_{n-1}$ with straight orientation and we will use our favorite notation for indecomposable $\Lambda$-modules with support at these vertices, namely, for $0\le i<j<n$, $M_{ij}$ will denote the module with support on the half open interval $(i,j]$. By Theorem \ref{thm: chord diagrams and exc seq of type An} (see Figure \ref{Fig: strand to chord diagram}), these correspond to unoriented chords on a circle with $n$ marked points ordered clockwise. As representations of $\Lambda$ (the algebra of type $C_n$), we will associate to $M_{ij}$, for $i<j$, the oriented chord $\gamma_{ij}$ on this circle.

In \cite{Ringel2}, it is shown that the enumeration of complete exceptional sequences for any Dynkin algebra (such as one of type $B_n$ or $C_n$) is independent of the orientation of the (modulated) quiver. This also follows from APR-tilting (Remark \ref{rem: about APR and orientation}). Also, it is pointed out in \cite{Ringel2} that the combinatorics of $B_n$ and $C_n$ are equivalent. However, to keep the notation simple, we stick with the straight orientation in type $C_n$.

With the straight descending orientation, the projective modules are contained in each other in the sequence:
\[
	P_1\subset P_2\subset \cdots\subset P_n
\]
For all $j\le n$, $P_j$ has dimension vector
\[
	\undim P_j=\alpha_1+\alpha_2+\cdots+\alpha_j
\]
where, for $i<n$, $\alpha_i$ is the $i$th simple root, which is also the $i$th unit vector in $\ZZ^n$ and the last root $\alpha_n$ is defined to be the sum of all positive simple roots. Thus
\[
	\alpha_n=(1,1,\cdots,1).
\]
This gives consistent notation for $\undim P_j$ for all $j\le n$. 

Note that $P_j=M_{0j}$ for $0<j<n$. For $P_n$ we use the notation $P_n=M_{00}$. This is one of the $n$ indecomposable $\Lambda$-modules with endomorphism ring the larger field $F$ instead of $K$. The others are $M_{ii}=\tau^{-i}P_n$ with dimension vector\footnote{In type $B_n$ the dimension vectors of $M_{ii}$ must be divided by $2$ and we take $\alpha_n=(1,\dots,1,2)$. The support of each $M_{ij}$ is the same as in the $C_n$ case. So, Lemma \ref{lem: Mij <-> gamma-ij respects noncrossing pairs} holds with the same proof.}
\[
	\undim M_{ii}=\alpha_{i+1}+\alpha_{i+2}+\cdots+\alpha_n-\alpha_1-\alpha_2-\cdots-\alpha_i.
\]
More generally, for $0\le j\le i<n$, let $M_{ij}$ denote the indecomposable $\Lambda$-module with dimension vector
\[
	\undim M_{ij}=\alpha_{i+1}+\alpha_{i+2}+\cdots+\alpha_n-\alpha_1-\alpha_2-\cdots-\alpha_j
\]
Mnemically, the modules $M_{ij}$, for all $i,j$, are sums of the roots $\alpha_k$ starting at $k=i+1$ and going up to $k=j$ modulo $n$ with the sign of $\alpha_k$ reversed when we go past $k=n$. One can see that the support of $M_{ij}$ for $j\le i$ is $(j+1,n]$.

The Coxeter element of the Weyl group for $\Lambda$, in terms of the basic $\alpha_i$, is
\[
	\mat{0 & I_{n-1} \\ -1 & 0}.
\]
This makes $\tau M_{ij}=M_{i-1,j-1}$ with indices modulo $n$ for $i>0$ (i.e., for $M_{ij}$ not projective) and $\tau M_{0j}=\tau P_j=I_j[-1]$ where $I_j=M_{n-1,j-1}$. This makes the Auslander-Reiten quiver of $\Lambda$ fit the same pattern as that of the bricks in the tube $\cW_n$ which we review below. For example, for $n=4$ we have the following

\begin{figure}[htbp]
\begin{center}
\begin{tikzpicture}[xscale=1.6,yscale=1.25]
\begin{scope}
\draw[dashed] (1,0)--(4,0);
\draw[dashed] (1.5,0.75)--(4.5,0.75);
\draw[dashed] (2,1.5)--(5,1.5);
\draw[dashed] (2.5,2.25)--(5.5,2.25);
\foreach \x/\y in {(1,0)/01,(2,0)/12,(3,0)/23,(4,0)/30}
\draw[fill,white] \x circle[radius=2mm];
\foreach \x/\y in {(1,0)/01,(2,0)/12,(3,0)/23,(4,0)/30}
\draw \x node{$M_{\y}$};
\foreach \x/\y in {(1.5,.75)/02,(2.5,0.75)/13,(3.5,0.75)/20,(4.5,0.75)/31}
\draw[fill,white] \x circle[radius=2mm];
\foreach \x/\y in {(1.5,.75)/02,(2.5,0.75)/13,(3.5,0.75)/20,(4.5,0.75)/31}
\draw \x node{$M_{\y}$};
\foreach \x/\y in {(2,1.5)/03,(3,1.5)/10,(4,1.5)/21,(5,1.5)/32}
\draw[fill,white] \x circle[radius=2mm];
\foreach \x/\y in {(2,1.5)/03,(3,1.5)/10,(4,1.5)/21,(5,1.5)/32}
\draw \x node{$M_{\y}$};
\foreach \x/\y in {(2.5,2.25)/00,(3.5, 2.25)/11,(4.5, 2.25)/22,(5.5, 2.25)/33}
\draw[fill,white] \x circle[radius=2mm];
\foreach \x/\y in {(2.5,2.25)/00,(3.5, 2.25)/11,(4.5, 2.25)/22,(5.5, 2.25)/33}
\draw \x node{$M_{\y}$};
\foreach \x/\y in {(0.4,0)/1,(0.9,.75)/2,(1.4,1.5)/3,(1.9,2.25)/4}
\draw \x node{$P_\y=$};
\foreach \x/\y in {(4.55,0)/1,(5.05,.75)/2,(5.55,1.5)/3,(6.05,2.25)/4}
\draw \x node{$=I_\y$};
\foreach \x in {1,...,4}\draw (\x.25,.4) node{$\nearrow$};
\foreach \x in {1,...,4}\draw (\x.75,1.1) node{$\nearrow$};
\foreach \x in {2,...,5}\draw (\x.25,1.8) node{$\nearrow$};
\foreach \x in {1,...,3}\draw (\x.75,.35) node{$\searrow$};
\foreach \x in {2,...,4}\draw (\x.25,1.1) node{$\searrow$};
\foreach \x in {2,...,4}\draw (\x.75,1.85) node{$\searrow$};
\end{scope}
\end{tikzpicture}
\caption{The Auslander-Reiten quiver for $C_4$. Auslander-Reiten translation $\tau$ reduces indices $i,j$ in $M_{ij}$ by 1 modulo $n=4$.}
\label{fig: AR quiver for C4}
\end{center}
\end{figure}

Recall that a $\Lambda$-module $E$ is called \emph{exceptional} if it is rigid and indecomposable. This holds if $E$ is isomorphic to one of the $n^2$ modules $M_{ij}$ discussed above. An \emph{exceptional sequence} for $\Lambda$ is a sequence $(E_1,\cdots,E_k)$ of exceptional modules $E_i$ so that, for all $1\le i<j\le k$ we have:
\[
	\Hom_\Lambda(E_j,E_i)=0=\Ext_\Lambda(E_j,E_i).
\]
An exceptional sequence of length $2$ is called an \emph{exceptional pair}. So, the required condition is that $(E_i,E_j)$ is an exceptional pair for all $1\le i<j\le k$.

\begin{lem}\label{lem: Mij <-> gamma-ij respects noncrossing pairs}
$(M_{ij},M_{ab})$ is an exceptional pair of $\Lambda$-modules if and only if $(\gamma_{ij},\gamma_{ab})$ is a noncrossing pair of objects in $\widetilde\cC_n$.
\end{lem}

\begin{proof}
It will suffice to prove this in the special case when $M_{ab}=P_k$, the $k$th projective module and $\gamma_{ab}=\gamma_{0k}$ if $k<n$ and $\gamma_{ab}$ is the loop $\gamma_{00}$ when $k=n$. The reason is that $(M_{ij},M_{ab})$ is an exception pair if and only if $(\tau M_{ij},\tau M_{ab})$ is an exceptional pair. We apply $\tau^a$ to get $\tau^aM_{ab}=M_{0k}=P_k$ where $k=b-a$ modulo $n$. (But beware the case $k=n$ where the notation is $P_n=M_{00}$.) Similarly, the notion of noncrossing pairs in $\widetilde\cC_n$ is defined geometrically and is thus invariant under rotation. Thus $(\gamma_{ij},\gamma_{ab})$ is a noncrossing pair if any only if $(\tau \gamma_{ij},\tau \gamma_{ab})$ is a noncrossing pair where $\tau$ is defined on all chords to be rotation to the left: $\tau\gamma_{ij}:=\gamma_{i-1,j-1}$ with indices reduced modulo $n$. We apply $\tau^a$ to get $\tau^a\gamma_{ab}=\gamma_{0k}$. 

Thus, it suffices to consider the special case $(a,b)=(0,k)$. In this case $(M_{ij},M_{0k})$ is an exceptional pair if and only if $k$ is not in the support of $M_{ij}$. There are 5 possibilities which we list. In the first two, $k$ is in the support of $M_{ij}$. In the last three it is not.
\begin{enumerate}
\item[(a)] $i<k\le j$. Then $k\in supp\,M_{ij}=(i,j]$.
\item[(b)] $k>j\le i$. Then $k\in supp\,M_{ij}=(j,n]$. This includes the case $k=n$ and $j\le i$.
\item[(c)] $k\le i<j$. Then $k\notin supp\,M_{ij}=(i,j]$.
\item[(d)] $i<j<k$. Then $k\notin supp\,M_{ij}=(i,j]$. This includes the case $k=n$ and $i<j$.
\item[(e)] $k\le j\le i$. Then $k\notin supp\,M_{ij}=(j,n]$.
\end{enumerate}
Figure \ref{Fig: Cases (a) and (b): crossing pairs} shows that $(\gamma_{ij},\gamma_{0k})$ is a crossing pair in cases (a) and (b). Case (b) includes the special case $(\gamma_{jj},\gamma_{00})$ which is crossing by definition since loops are not compatible with each other. Figure \ref{Fig: Cases (c,d,e): noncrossing pairs} shows that $(\gamma_{ij},\gamma_{0k})$ is a noncrossing pair in cases (c), (d) and (e). 
\end{proof}

\begin{figure}[htbp]
\begin{center}
\begin{tikzpicture}[scale=.8]
\begin{scope}  
\coordinate (A0) at (0,-2);
\coordinate (A1) at (-1.73,-1);
\coordinate (A2) at (-1.73,1);
\coordinate (A3) at (0,2);
\coordinate (A13) at (-.87,.5);
\coordinate (A4) at (1.73,1);
\coordinate (A14) at (0.87,.5);
\coordinate (A5) at (1.73,-1);
\coordinate (X) at (.4,-1);
\coordinate (Y) at (-1.3,.8);
\coordinate (Z) at (1.2,.2);
\draw (X) node{$\gamma_{0k}$};
\draw (Y) node{$\gamma_{ik}$};
\draw (Z) node{$\gamma_{ij}$};
\coordinate (C) at (0,-3);
\draw (C) node{Case (a)};
\draw[thick] (0,0) circle[radius=20mm];
\foreach \x in {0,1,3,4}
\draw[fill] (A\x) circle[radius=1mm];
\draw[very thick,->] (A3)--(A13) (A1)--(A13);
\draw[very thick,->](A0)--(A3) (A0)--(0,-.3) ;
\draw[very thick, ->] (A1)--(A4) (A1)--(A14);
\draw (A1) node[left]{$i$};
\draw (A4) node[right]{$j$};
\draw (A0) node[below]{$0$};
\draw (A3) node[above]{$k$};
\end{scope}
\begin{scope}[xshift=7cm]  
\coordinate (A0) at (0,-2);
\coordinate (A1) at (-1.73,-1);
\coordinate (A2) at (-1.73,1);
\coordinate (A3) at (0,2);
\coordinate (A13) at (-.87,.5);
\coordinate (A4) at (1.73,1);
\coordinate (A14) at (0.87,.5);
\coordinate (A5) at (1.73,-1);
\coordinate (X) at (.4,-1);
\coordinate (C) at (0,-3);
\draw (C) node{Case (b)};
\draw (X) node{$\gamma_{0k}$};
\draw[thick] (0,0) circle[radius=20mm];
\foreach \x in {0,1,2,3,4}
\draw[fill] (A\x) circle[radius=1mm];
\draw[very thick,->] (A2)--(A1) (A2)--(-1.73,0);
\draw[very thick,->](A0)--(A3) (A0)--(0,-.3) ;
\draw[very thick, ->] (A1)--(A4) (A4)--(A14);
\draw (A1) node[left]{$j$};
\draw (A2) node[left]{$i$};
\draw (A4) node[right]{$i'$};
\draw (A0) node[below]{$0$};
\draw (A3) node[above]{$k$};
\end{scope}
\end{tikzpicture}
\caption{In Case (a) at left, $(\gamma_{ik},\gamma_{0k})$ is not compatible since $\gamma_{ik}$ is clockwise from $\gamma_{0k}$. $(\gamma_{ij},\gamma_{0k})$ are crossing. In Case (b) on the right, $\gamma_{ij}$ and $\gamma_{0k}$ are not compatible since they go counterclockwise around the region between them. $(\gamma_{i'j},\gamma_{0k})$ are crossing.}
\label{Fig: Cases (a) and (b): crossing pairs}
\end{center}
\end{figure}
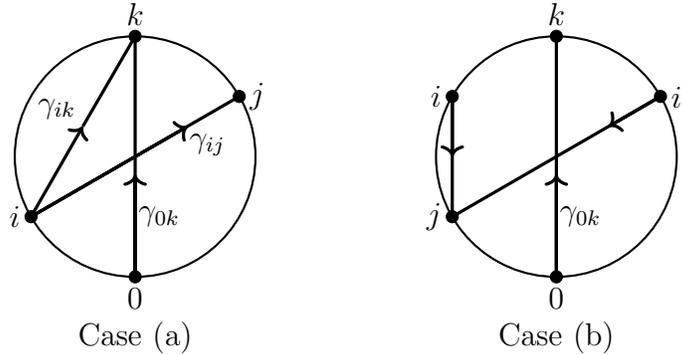
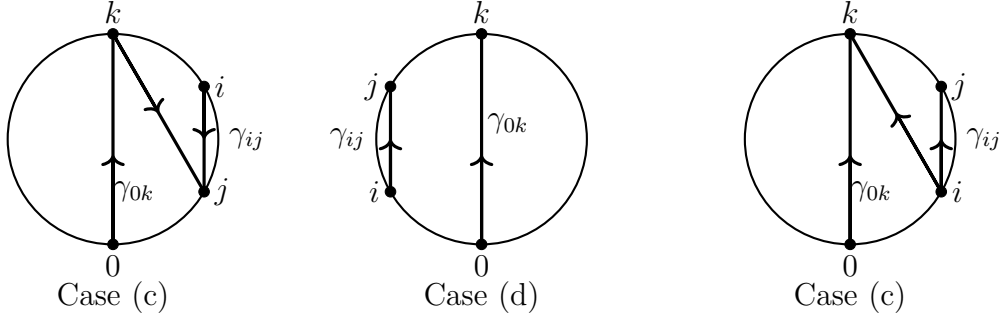
\begin{figure}[htbp]
\begin{center}
\begin{tikzpicture}[scale=.7]
\begin{scope}  
\coordinate (A0) at (0,-2);
\coordinate (A1) at (-1.73,-1);
\coordinate (A2) at (-1.73,1);
\coordinate (A3) at (0,2);
\coordinate (A53) at (.87,.5);
\coordinate (A4) at (1.73,1);
\coordinate (A14) at (0.87,.5);
\coordinate (A5) at (1.73,-1);
\coordinate (X) at (.4,-1);
\coordinate (Y) at (-1.3,.8);
\coordinate (Z) at (2,0);
\draw (X) node{$\gamma_{0k}$};
\draw (Z) node[right]{$\gamma_{ij}$};
\coordinate (C) at (0,-3);
\draw (C) node{Case (c)};
\draw[thick] (0,0) circle[radius=20mm];
\foreach \x in {0,5,3,4}
\draw[fill] (A\x) circle[radius=1mm];
\draw[very thick,->] (A3)--(A5) (A3)--(A53);
\draw[very thick,->](A0)--(A3) (A0)--(0,-.3) ;
\draw[very thick, ->] (A5)--(A4) (A4)--(1.73,0);
\draw (A5) node[right]{$j$};
\draw (A4) node[right]{$i$};
\draw (A0) node[below]{$0$};
\draw (A3) node[above]{$k$};
\end{scope}

\begin{scope}[xshift=7cm] 
\coordinate (A0) at (0,-2);
\coordinate (A1) at (-1.73,-1);
\coordinate (A2) at (-1.73,1);
\coordinate (A3) at (0,2);
\coordinate (A13) at (-.87,.5);
\coordinate (A4) at (1.73,1);
\coordinate (A14) at (0.87,.5);
\coordinate (A5) at (1.73,-1);
\coordinate (X) at (.5,.3);
\coordinate (C) at (0,-3);
\draw (C) node{Case (d)};
\draw (X) node{$\gamma_{0k}$};
\coordinate (Z) at (-2,0);

\draw (Z) node[left]{$\gamma_{ij}$};

\draw[thick] (0,0) circle[radius=20mm];
\foreach \x in {0,1,2,3}
\draw[fill] (A\x) circle[radius=1mm];
\draw[very thick,->] (A2)--(A1) (A1)--(-1.73,0);
\draw[very thick,->](A0)--(A3) (A0)--(0,-.3) ;
\draw (A1) node[left]{$i$};
\draw (A2) node[left]{$j$};
\draw (A0) node[below]{$0$};
\draw (A3) node[above]{$k$};
\end{scope}
\begin{scope}[xshift=14cm] 
\coordinate (A0) at (0,-2);
\coordinate (A1) at (-1.73,-1);
\coordinate (A2) at (-1.73,1);
\coordinate (A3) at (0,2);
\coordinate (A53) at (.87,.5);
\coordinate (A4) at (1.73,1);
\coordinate (A14) at (0.87,.5);
\coordinate (A5) at (1.73,-1);
\coordinate (X) at (.4,-1);
\coordinate (Y) at (-1.3,.8);
\coordinate (Z) at (2,0);
\draw (X) node{$\gamma_{0k}$};
\draw (Z) node[right]{$\gamma_{ij}$};
\coordinate (C) at (0,-3);
\draw (C) node{Case (c)};
\draw[thick] (0,0) circle[radius=20mm];
\foreach \x in {0,5,3,4}
\draw[fill] (A\x) circle[radius=1mm];
\draw[very thick,->] (A3)--(A5) (A5)--(A53);
\draw[very thick,->](A0)--(A3) (A0)--(0,-.3) ;
\draw[very thick, ->] (A5)--(A4) (A5)--(1.73,0);
\draw (A5) node[right]{$i$};
\draw (A4) node[right]{$j$};
\draw (A0) node[below]{$0$};
\draw (A3) node[above]{$k$};
\end{scope}
\end{tikzpicture}
\caption{In Case (c) at left, $\gamma_{ij},\gamma_{0k}$ are compatible in either order since they clockwise around the region between. $(\gamma_{kj},\gamma_{0k})$ is noncrossing since $\gamma_{ok}$ is clockwise from $\gamma_{kj}$. In Case (d) $(\gamma_{ij},\gamma_{0k})$ are crossing. In Case (d), $\gamma_{ij}$ and $\gamma_{0k}$ are compatible in either order. In Case (e). $\gamma_{ij},\gamma_{0k}$ are noncrossing in either order. $(\gamma_{ik},\gamma_{0k})$ is a noncrossing pair since $\gamma_{0k}$ is clockwise from $\gamma_{ik}$.}
\label{Fig: Cases (c,d,e): noncrossing pairs}
\end{center}
\end{figure}

Lemma \ref{lem: Mij <-> gamma-ij respects noncrossing pairs} implies the following.

\begin{thm}\label{thm: exc seq for Cn = nonX for chords and loops}
The correspondence $M_{ij}\leftrightarrow \gamma_{ij}$ induces a bijection:
\[
	\left\{
	\begin{matrix}\text{exceptional sequences of length $k$}\\
	\text{for $\Lambda$ an algebra of type $C_n$}
	\end{matrix}
	\right\}
	\xrightarrow\cong
	\left\{
	\begin{matrix}\text{length $k$ noncrossing sequences of}\\
	\text{loops and oriented chords in $\widetilde\cC_n$}
	\end{matrix}
	\right\}
\]
\end{thm}

We observe that exceptional sequences of type $C_n$ are equivalent to weak exceptional sequences for the Nakayama algebra given an $n$-cycle modulo $rad^n=0$. These in turn are $\tau$-exceptional sequences \cite{BM} for the Nakayama algebra. See \cite{Sen2}.

\subsection{Soft exceptional sequences in the tube}\label{ss: soft exc seq in tubes}

As we said in the introduction, we take $\cW_n$, the abelian tube of rank $n$. Hereditary algebras of tame type all have tubes. These are $\tau$-periodic. Every module in the tube has a $\tau$ orbit with $n$ objects. Also, all tubes of rank $n$ are essentially isomorphic. The objects form the category of nilpotent representations of an oriented $n$ cycle. If necessary to be more specific, we take the exceptional tube of rank $n$ for the affine quiver algebra of type $\tilde A_n$ given by the quiver
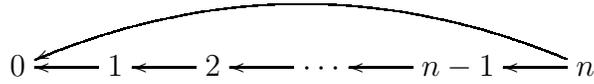
\[
\xymatrixrowsep{10pt}
\xymatrix{
\ \\
0  & 1\ar[l] & 2\ar[l] & \cdots\ar[l] & n-1\ar[l]&
	n \ar[l] \ar@/_2pc/[lllll]} 
\]
As a subcategory of the module category of a tame algebra, the category $\cW_n$ is the extension closure of the quasi-simple modules at the mouth of the tube. Although it is given by a finite number of objects (the $n$ quasi-simple objects) it is an abelian category which is not finitely generated. This means there is no object $M$ in $\cW_n$ so that all other objects are quotients of direct sums of $M$ with itself. If this were true, all objects of $\cW_n$ would have Loevy-length bounded by that of $M$. However, the indecomposable objects of $\cW_n$ are uniserial with unbounded Loevy-length. $\cW_n$ has no projective objects. But $\cW_n$ is given by its $n$ simple objects and we say that it is \emph{spanned} by these objects.

In the $\tilde A_n$ case, these quasi-simple modules are $S_1,S_2,\cdots,S_{n-1}$ and the regular module with support at vertices $0,n$. The extension closure of the objects $S_i$ form a category of type $A_{n-1}$ which we call $\widetilde A_{n-1}$. However, tubes are rotationally symmetric since Auslander-Reiten translation $\tau$ rotates the tube. Thus, any $n-1$ consecutive quasi-simple modules will {span} a category of type $A_{n-1}$. 

We denote by $S_1,\cdots,S_n$ the simple objects of the category $\cW_n$. The quiver of $\cW_n$ is a single oriented cycle of length $n$ with descending orientation: 
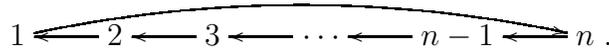
\[
\xymatrixrowsep{20pt}
\xymatrix{
1 \ar@/^1pc/[rrrrr] & 2\ar[l] & 3\ar[l] & \cdots\ar[l] & n-1\ar[l]&
	n \ar[l]}. 
\]
Since $\cW_n$ is closed under $\tau$, we have by Auslander-Reiten duality that $\Ext(X,Y)$ is the vector space dual of $\Hom(Y,\tau X)$:
\[
	\Ext(X,Y)\cong D\Hom(Y,\tau X).
\]

The tube $\cW_n$ has an infinite number of indecomposable objects which are uniserial and uniquely determined up to isomorphism by their length and top. We denote these by $W_{ij}$ where $i<j$ are integers. This is the object of length $j-i$ with top $S_j$. The structure of these is well-known: $\tau W_{ij}=W_{i-1,j-1}$. Since all objects are $\tau$-periodic with period $n$, we have $W_{ij}=W_{i+pn,j+pn}$ for any integer $p$. The support of $\Hom$ is easy to determine: 
\[
\Hom(W_{ab},W_{ij})\neq 0\quad\text{ if }\quad a\le i+pn<b\le j+pn \text{ for some $p\in\ZZ$}.
\]
In particular, $W_{ij}$ is a \emph{brick}, i.e., its endomorphism ring is $K$ if and only if its length is $\le n$. Since $\Hom(W_{ab},W_{ij})\cong \Hom(\tau W_{ab},\tau W_{ij})$, we can restrict to the case when $a=0,b=k>0$ (and apply $\tau$, rotating the figures, to get all the other cases). Then, for $0<j,k\le n$ we have:
\begin{enumerate}
\item[(a)] $\Hom(W_{0k},W_{ij})\neq0$ if $0\le i<k\le j$.
\item[(b)] $\Ext(W_{0k},W_{ij})\cong D\Hom(W_{ij},\tau W_{0k})\neq0$ if $i<0\le j<k$.
\end{enumerate}
This is a well-known and often visualized in a diagram similar to Figure \ref{Fig: supports of hom and ext in Wn}.
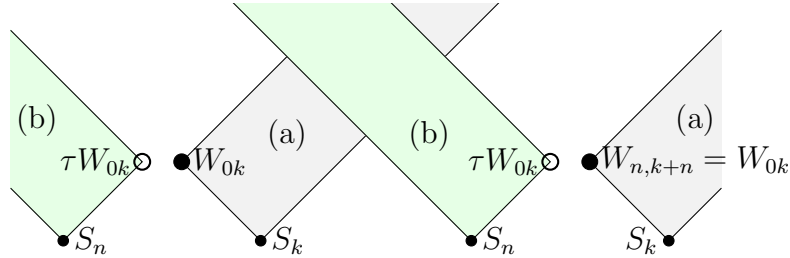
\begin{figure}[htbp]
\begin{center}
\begin{tikzpicture}[scale=.7]

\clip (-1,-1) rectangle (16,4.5);

\coordinate (A) at (0,1);
\coordinate (B) at (0,4);
\coordinate (S0) at (1,0);
\coordinate (TW1) at (2.5,1.5);
\coordinate (BB2) at (0.5,2.3);

\draw[fill,green!10!white] (A)--(B)--(TW1)--(S0)--(A);
\draw (A)--(S0)--(TW1)--(B);
\draw (BB2) node{(b)};
\draw[thick] (TW1) circle[radius=1.5mm];

\begin{scope}[xshift=-2.5mm]
\coordinate (E) at (7.5,5.5);
\coordinate (F) at (10.5,5.5);
\coordinate (W1) at (3.5,1.5);
\coordinate (Sk) at (5,0);
\coordinate (AA) at (5.5,2);
\coordinate (BB) at (8.2,2);
\draw[fill,gray!10!white] (F)--(E)--(W1)--(Sk)--(F);
\draw (E)--(W1)--(Sk)--(F);
\draw[fill] (W1) circle[radius=1.5mm];

\coordinate (C) at (3.5,5.5);
\coordinate (D) at (6.5,5.5);
\coordinate (Sn) at (9,0);
\coordinate (TW2) at (10.5,1.5);
\draw[fill,green!10!white] (C)--(D)--(TW2)--(Sn)--(C);
\draw (D)--(TW2)--(Sn)--(C);
\draw (AA) node{(a)};
\draw (BB) node{(b)};
\draw[thick] (TW2) circle[radius=1.5mm];
\end{scope}

\begin{scope}[xshift=-5mm]
\coordinate (W2) at (11.5,1.5);
\coordinate (S2) at (13,0);
\coordinate (G) at (14,4);
\coordinate (H) at (14,1);
\coordinate (AA2) at (13.5,2.3);
\draw[fill,gray!10!white] (H)--(G)--(W2)--(S2)--(H);
\draw (G)--(W2)--(S2)--(H);
\draw (AA2) node{(a)};
\draw[fill] (W2) circle[radius=1.5mm];
\end{scope}

\foreach \x in {S0,Sk,Sn,S2}
\draw[fill] (\x) circle[radius=1mm];
\draw (S0) node[right]{$S_n$};
\draw (Sn) node[right]{$S_n$};
\draw (Sk) node[right]{$S_k$};
\draw (S2) node[left]{$S_k$};
\draw (W1) node[right]{$W_{0k}$};
\draw (W2) node[right]{$W_{n,k+n}=W_{0k}$};
\draw (TW1) node[left]{$\tau W_{0k}$};
\draw (TW2) node[left]{$\tau W_{0k}$};
\end{tikzpicture}
\caption{(a) The support of $\Hom(W_{0k},-)$ is shown in gray. 
(b) the support of $\Ext(W_{0k},-)$ is shown in green.}
\label{Fig: supports of hom and ext in Wn}
\end{center}
\end{figure}

\begin{defn}
We define a \emph{soft exceptional sequence} in $\cW_n$ to be a sequence of bricks $(W_1,\cdots,W_k)$ so that $\Hom(W_j,W_i)=0=\Ext(W_j,W_i)$ for $1\le i<j\le k$. A soft exceptional sequence of length 2 is called an \emph{soft exceptional pair}. A soft exceptional sequence is called an \emph{exceptional sequence} if all of its objects are rigid.
\end{defn}

There are $n^2$ bricks in $\cW_n$ and we denote them by $V_{ij}$ for $0\le i,j<n$. These are:
\[
	V_{ij}=\begin{cases} W_{ij} & \text{if } i<j\\
   W_{i,j+n} & \text{if } i\ge j
    \end{cases}
\]
$V_{ij}$ is rigid if and only if $i\neq j$. Figure \ref{fig: bricks in Wn} shows the bricks in a tube of rank $4$.

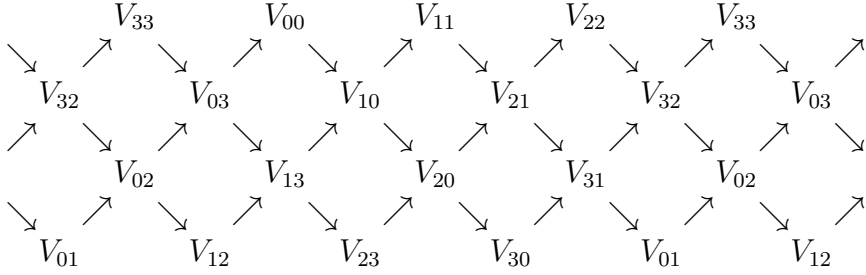
\begin{figure}[htbp]
\begin{center}
\begin{tikzpicture}[xscale=1.9,yscale=1.3]
\draw[dashed] (-.5,0)--(5.5,0);
\foreach \x/\y in {(0,0)/01,(1,0)/30,(2,0)/01,(3,0)/12,(4,0)/23,(5,0)/30}
\draw[fill,white] \x circle[radius=2mm];

\foreach \x/\y in {(0,0)/01,(1,0)/12,(2,0)/23,(3,0)/30,(4,0)/01,(5,0)/12}
\draw \x node{$V_{\y}$};
\foreach \x/\y in {(0.5,.75)/02,(1.5,.75)/13,(2.5,0.75)/20,(3.5,0.75)/31,(4.5,0.75)/02}
\draw \x node{$V_{\y}$};
\foreach \x/\y in {(0.5,2.25)/33,(1.5,2.25)/00,(2.5,2.25)/11,(3.5,2.25)/22,(4.5,2.25)/33}
\draw \x node{$V_{\y}$};
\foreach \x/\y in {(0,1.5)/32,(1,1.5)/03,(2,1.5)/10,(3,1.5)/21,(4,1.5)/32,(5,1.5)/03}
\draw \x node{$V_{\y}$};
\begin{scope}[xshift=-1cm]
\foreach \x in {1,...,6}\draw (\x.25,.4) node{$\nearrow$};
\foreach \x in {0,...,5}\draw (\x.75,1.1) node{$\nearrow$};
\foreach \x in {1,...,6}\draw (\x.25,1.9) node{$\nearrow$};
\foreach \x in {0,...,5}\draw (\x.75,.35) node{$\searrow$};
\foreach \x in {1,...,6}\draw (\x.25,1.1) node{$\searrow$};
\foreach \x in {0,...,5}\draw (\x.75,1.85) node{$\searrow$};
\end{scope}
\end{tikzpicture}
\caption{Bricks in the tube of rank 4. These are the objects at the mouth of the tube $\cW_4$. The objects in the top row are nonrigid bricks which all have the same dimension vector.}
\label{fig: bricks in Wn}
\end{center}
\end{figure}

\begin{lem}\label{lem: compatibility in Wn is same as in mod Cn}
$(V_{ij},V_{ab})$ is a soft exceptional pair in $\cW_n$ if and only if $(M_{ij},M_{ab})$ is an exceptional sequence in $mod\text-\Lambda$ for $\Lambda$ of type $C_n$.
\end{lem}

\begin{proof} By applying $\tau^a$ we may assume $a=0$. Then, in the notation of $V_{ij}$, Conditions (a), (b) above are:
\begin{enumerate}
\item[(a)] $\Hom(V_{0k},V_{ij})\neq0$ if either $0\le i<k\le j$ or $j\le i<k$.
\item[(b)] $\Ext(V_{0k},V_{ij})\neq0$ if $j<k$ and $i\ge j$.
\end{enumerate}
the union of these two conditions, which describe all the cases when $(V_{ij},V_{0k})$ is not a soft exceptional pair is the same as the union of cases (a) and (b) in the proof of Lemma \ref{lem: Mij <-> gamma-ij respects noncrossing pairs} which are all cases where $(M_{ij},M_{0k})$ is not an exceptional pair. Thus, the compatibility conditions for $V_{ij}$ is the same as those for $M_{ij}$.
\end{proof}

This leads us to the proof of Theorem \ref{thm A}.

\subsection{Three versions of exceptional sequences of length $k$}\label{ss: three versions}

We start with the bijections
\[
	M_{ij}\leftrightarrow V_{ij} \leftrightarrow \gamma_{ij}
\]
between the following sets, each having $n^2$ elements.

(a) $\cE_n=$ the set of all (isomorphism classes of) exceptional objects $M_{ij}$ in $mod\text-\Lambda$ for $\Lambda$ an hereditary algebra of type $C_n$.

(b) $\cB_n$ is the set of all bricks in $\cW_n$:
\[
	\cB_n=\{V_{ij}\,|\, 0\le i,j<n \}.
\]

(c) $\widetilde\cC_n$ is the set of all loops and oriented chords in a circle with $n$ marked points:
\[
	\widetilde\cC_n=\{\gamma_{ij}\,|\, 0\le i,j<n\}.
\]
We assume the marked points are labeled $0$ through $n-1$ in clockwise order around the circle. This corresponds to the descending orientation of the modulated quiver for $\Lambda$.

\begin{thm}\label{thm: k-element version of Thm A}
For any $k\le n$, the bijection between these three sets given by $M_{ij}\leftrightarrow V_{ij} \leftrightarrow \gamma_{ij}$ gives a bijection between the following three sets.

\emph{(A)} $\cE_n^{(k)}=$ the set of exceptional sequences of length $k$ in $mod\text-\Lambda$.

\emph{(B)} $\cB_n^{(k)}=$ the set of all soft exceptional sequences of length $k$ in $\cW_n$.

\emph{(C)} $\widetilde\cC_n^{(k)}=$ is the set of all noncrossing sequences of $k$ elements of $\widetilde\cC_n$.
\end{thm}

\begin{proof}Lemma \ref{lem: compatibility in Wn is same as in mod Cn} implies that the correspondence $V_{ij}\leftrightarrow M_{ij}$ gives a bijection between (A) and (B). Theorem \ref{thm: exc seq for Cn = nonX for chords and loops} says that the bijection $M_{ij}\leftrightarrow \gamma_{ij}$ gives a bijections between (A) and (C). \end{proof}

This gives most of Theorem \ref{thm A}. The size of these sets is given in Corollary \ref{cor: no of exc seq for Bn/Cn}. It remains to give a bijection between the three sets $\cE_n^{(n)}$, $\cB_n^{(n)}$ and $\widetilde \cC_n^{(n)}$ and the set $\widetilde\cT_n$ of augmented rooted labeled trees.
 

\subsection{Bijection with $\widetilde\cT_n$}\label{ss: bijection with augmented trees}

The last bijection is with $\widetilde \cT_n$, the set of augmented rooted labeled trees. We need to review the bijection from our previous paper, with the orientation of the quiver reversed.

\begin{thm}\cite{Ig-Sen}\label{thm: An case}
There is a bijection between 

$\cF_{n}=$ the set of rooted labeled forests with $n$ vertices and

$\cE_{n}^A=$ the set of complete exceptional sequences for linear $A_{n}$: 
\[
1\leftarrow 2\leftarrow\cdots\leftarrow n.
\]
Furthermore, this bijection is given by associating to a complete exceptional sequence $(E_1,\cdots,E_{n})$ the Hasse diagram of the $E_i$ partially ordered by inclusion of supports.
\end{thm}

The bijection $\cF_n\cong \cE_n^A$ restricts to a bijection between $\cT_n$, the set of rooted labeled trees $T$ with $n$ vertices and $\cA_n$, the set of complete exceptional sequences $E_\ast=(E_1,\cdots,E_n)$ for linear $A_n$ which include the projective-injective $P_n=I_1=M_{0n}$.

The bijection $\widetilde \cT_n\cong \cB_n^{(n)}$ is given by sending $(T,k)$ to $\tau^{-k}E_\ast$ where $E_\ast$ is the exceptional sequence corresponding to $T$. The bijection $\widetilde \cT_n\cong \widetilde\cC_n^{(n)}$ is easier to describe since $\tau$ just rotates the circle and the support of an oriented chord $\gamma_{ij}$ is defined to be the set of arcs on the circle on the left side of $\gamma_{ij}$ which, in the case $i<j$, is the half open interval $(i,j]$. 

\begin{thm}\label{thm: A with proof}
There are bijections:
\[
	\cE_n^{(n)}\cong \cB_n^{(n)}\cong\widetilde \cC_n^{(n)}\cong\widetilde\cT_n
\]
where the first three sets are in bijection by Theorem \ref{thm: k-element version of Thm A} and the bijection with $\widetilde \cT_n$ is given by taking $(\gamma_1,\cdots,\gamma_n)\in \widetilde \cC_n^{(n)}$ to $(T,\ell)$ where $T$ is the Hasse diagram of the set $\{\gamma_i\}$ ordered by inclusion of supports and $\ell$ is the marked point where the loop is.
\end{thm}

Thus, if the $k$-th term of $(\gamma_1,\cdots,\gamma_n)$ is the loop $\gamma_k=\gamma_{\ell\ell}$, the augmentation of the corresponding augmented tree is $\varepsilon =\ell$. The number $k$ means that the $k$th vertex $v_k$ is the root of $T$.

\begin{proof}
The correspondence described in the theorem commutes with rotation of the pointed chord diagram and reduction of the augmentation $\varepsilon=\ell$. This reduces to the case $\varepsilon=0$ which follows from Theorem \ref{thm: An case} (restricted to $\cT_n\subset\cF_{n-1}$).
\end{proof}

\subsection{Proof of Theorem \ref{thm B}}\label{ss: proof of Thm B}

There is a bijection between $\widetilde\cF_n$, the set of augmented rooted labeled forests with $n-1$ vertices and $\cR_n^{(n-1)}$, the set of complete exceptional sequences for $\cW_n$. The proof is that both are in bijection with a third set: $\cC_n^{(n-1)}$, the set of complete oriented chord diagrams in a circle with $n$ marked points. 

The bijection $\cR_n^{(n-1)}\cong \cC_n^{(n-1)}$ is elementwise. 

\begin{thm}\label{thm: bijection Rn=Cn}
The obvious bijection $V_{ij}\leftrightarrow \gamma_{ij}$ between the sets:

$\cR_n=$ the set of rigid objects in $\cW_n$ and

$\cC_n=$ the set of oriented chords in the circle with $n$ marked points

\noindent gives a bijection, for any $k<n$, between

$\cR_n^{(k)}=$ the set of exceptional sequences of length $k$ in $\cW_n$ and

$ \cC_n^{(k)}=$ the set of noncrossing chord diagrams of length $k$ in the circle with $n$ marked points.
\end{thm}

The cardinality of these sets was given in Theorem \ref{thm: cardinality of Cn(k)}.

\begin{proof}
We have shown, in Lemmas \ref{lem: Mij <-> gamma-ij respects noncrossing pairs} and \ref{lem: compatibility in Wn is same as in mod Cn}, that the bijection $V_{ij}\leftrightarrow \gamma_{ij}$ respects pairwise compatibility.
\end{proof}

To finish the proof of Theorem \ref{thm B} we take $k=n-1$. The bijection $\cR_n^{(n-1)}\cong \cC_n^{(n-1)}$ is given by Theorem \ref{thm: bijection Rn=Cn} above. The bijection $\widetilde\cF_{n-1}\cong \cC_n^{(n-1)}$ is induced by the epimorphism
\[
	\widetilde \cT_n\cong \widetilde\cC_n^{(n)}\onto \cC_n^{(n-1)}
\] 
where the map $\widetilde\cC_n^{(n)}\onto \cC_n^{(n-1)}$ is given by deleting the loop. Since deleting the loop in a pointed chord diagram corresponds to deleting the root of the corresponding augmented tree, there is a unique induced bijection $\widetilde\cF_{n-1}\cong \cC_n^{(n-1)}$ making the following diagram commute.
\[
\xymatrix{
\widetilde \cT_n\ar[d]\ar[r]^{\cong} &
	\widetilde\cC_n^{(n)}\ar[d]\\
\widetilde\cF_{n-1}\ar[r]^\cong &\cC_n^{(n-1)}
	} 
\]

\begin{thm}\label{thm B with proof}
There are bijections
\[
	\cR_n^{(n-1)}\cong \cC_n^{(n-1)}\cong \widetilde\cF_{n-1}
\]
as described above.
\end{thm}

Since we now have a good description of all of these bijections we can compare the action of the braid group $B_n$ on the sets in Theorem \ref{thm: A with proof}. The proof will use the bijections in Theorem \ref{thm B with proof} above.

\begin{cor}\label{cor: bijection is Bn equivariant}
The bijection between augmented rooted labeled tree and complete exceptional sequences for an algebra of type $C_n$ respects the action of the braid group $B_n$.
\end{cor}

\begin{proof} Recall from \cite{RingelExcSeq} and \cite{Crawley-Boevey} that the action of each of the generators $\sigma_i$ of the braid group on a complete exceptional sequence $E_\ast=(E_1,\cdots,E_n)$ over any hereditary algebra is given by deleting $E_{i+1}$ from the sequence and inserting a new term $E_i'$ before $E_i$:
\[
	\sigma_i E_\ast=(E_1,\cdots,E_{i-1},E_i',E_i,E_{i+2},\cdots,E_n).
\]
Since $E_i'$ is uniquely determined by the other terms in the sequence, $\sigma_i E_\ast$ is well-defined. 

There is a corresponding action of the braid group on pointed chord diagrams:
\[
	\sigma_i(\gamma_1,\cdots,\gamma_n)=(\gamma_1,\cdots,\gamma_{i-1},\gamma_i',\gamma_i,\gamma_{i+2},\cdots,\gamma_n)
\]
where $\gamma_i'$ is the unique object of $\widetilde\cC_n$ which completes the pointed chord diagrams. If $\beta$ denotes the bijection $\beta:\cE_n\cong \widetilde\cC_n$ and $\gamma_i=\beta(E_i)$ for each $i$ then $\gamma_i'$ is given by $\gamma_i'=\beta(E_i')$.

In \cite{Ig-Sen} we defined the action of $\sigma_i$ on forests to corresponds to the action on complete exceptional sequences for $A_n$. We are restricting this action to $\cT_n\subset\cF_n$. However, $\cT_n$ is not invariant under the action of the Braid group. $\sigma_i T\in \cT_n$ except possibly in the special case when the $i+1$st vertex of $T$ is the root $v_{i+1}=v_r$. In that case, we have the special rule that $\sigma_i T$ is simply $T$ with vertices $v_i,v_{i+1}$ in the other order, i.e., the root of $T$ becomes its $i$th vertex. We also recall that, for the augmented tree $(T,\varepsilon)$, the action of $\sigma_i$ does not change the augmentation except in this special case where we have:
\[
	\sigma_i(T,\varepsilon)=(\sigma_iT,\varepsilon').
\]
We will not need the formula for $\varepsilon'$ since it is uniquely determined by the fact that $(T,\varepsilon)$ and $(\sigma_iT,\varepsilon')$ have the same underlying augmented forest. By Theorem \ref{thm B with proof} above, the pointed chord diagrams corresponding to $(T,\varepsilon)$ and $(\sigma_iT,\varepsilon')$ have the same underlying oriented chord diagram. The only thing different is the location and position of the loop. For $(T,\varepsilon)$ the root is $\gamma_{\varepsilon\varepsilon}$ which is in position $i+1$ in the corresponding pointed chord diagram. For $\sigma_i(T,\varepsilon)=(\sigma_iT,\varepsilon')$, the loop is in position $i$. It is the unique object which fits in the $i$th position of the sequence of oriented chords and loops. Therefore, by definition of the braid group action, the new pointed chord diagram is $\sigma_i$ applied to the old one. This is also true when $v_{i+1}$ is not the root of $T$ since, in that case by \cite{Ig-Sen}, $\sigma_i$ acts on $T$ in the way corresponding to the action of $\sigma_i$ on exceptional sequences. In the corresponding pointed chord diagram, the loop does not move, we delete the oriented chord $\gamma_{i+1}$ and insert a new oriented chord in position $i$ of the pointed chord diagram. Thus the bijection $\widetilde \cC_n^{(n)}\cong\widetilde\cT_n$ is $B_n$-equivariant.

Since the bijections $\cE_n^{(n)}\cong \cB_n^{(n)}\cong\widetilde \cC_n^{(n)}$ are elementwise, the action of the braid group on these three sets is given by the same description, i.e., $\sigma_i$ deletes the $i+1$st object and inserts the unique possible new $i$th object. So, the action of $B_n$ on all of four sets agree. This proves the corollary.
\end{proof}

By definition, the action of the braid group on $\cT_n$ and $\widetilde\cT_n$ agree, i.e., the projection map $\widetilde \cT_n\to \cT_n$ is $B_n$-equivariant. This is an example of a more general phenomenon. Given the action of any group $G$ on any set $X$ and given any normal subgroup $N$ of $G$, it is easy to see that there is an induced action of $G$ on $X/N$, the set of orbits in $X$ of the action of $N$. For $G=B_n$, the center of $B_n$ is generated by $\Delta^2=\delta_n^n$ \cite{Garside}. By Theorem \ref{thm: action of delta n on tilde Tn}, this acts on augmented trees by $\Delta^2(T,\varepsilon)=(T,\varepsilon-1)$. Therefore, the set of orbits of the action of $\Delta^2=\delta_n^n$ on $\widetilde\cT_n$ is $\cT_n$ and we understand why the map $\widetilde\cT_n\to\cT_n$ is $B_n$-equivariant.


\subsection{Exceptional sets}\label{ss: exceptional sets}

In \cite{IM} the concept of \emph{exceptional set} is introduced. This is defined to be a set of objects in $mod\text-\Lambda$ (or $\cW_n$) which can be ordered to form a complete exceptional sequence. In \cite[Theorem 4.3]{IM} it was shown that the notion of a object being relatively projective in an exceptional sequence is independent of the order of the sequence. Thus, signed exceptional sets can also be defined. Analogous to this we could define (unordered) chord sets, oriented chord sets and pointed chord sets. The following was shown in \cite[Theorem 5.4]{IM}, but it was first shown in \cite{Araya}.

\begin{thm}\cite{Araya}\label{thm: number of exc sets for An}
The number of exceptional set for $A_n$ is 
\[
	\frac{(3n)!}{n!(2n+1)!}=\frac1{3n+1}\binom{3n+1}n.
\]
\end{thm}

Using Theorems \ref{thm A} and \ref{thm B} we can deduce the number of exceptional sets for $\cW_n$ and for algebras of type $C_n$.

\begin{cor}\label{cor: exc sets in Wn}
The number of exceptional sets in $\cW_n$ is
\[
	\frac{n(3n-3)!}{(n-1)!(2n-1)!}.
\]
\end{cor}

\begin{proof} We already know that exceptional sets for $A_{n-1}$ are in bijection with unordered complete chord diagrams in a circle with $n$ marked points. By Theorem \ref{thm: number of exc sets for An} above, there are $\frac{(3n-3)!}{(n-1)!(2n-1)!}$ such diagrams. By Theorem \ref{thm B with proof} above, exceptional sets for $\cW_n$ are in bijection with unordered complete oriented chord diagrams in a circle with $n$ marked points. As we observed in the proof of Theorem \ref{thm: cardinality of Cn(k)}, these are given by unoriented chord diagrams together with a choice of central region. Since there are $n$ regions in the complement of the $n-1$ chords, we multiply the number of unordered unoriented chord diagrams by $n$. This proves the theorem.
\end{proof}

\begin{cor}\label{cor: exc sets of Cn}
The number of exceptional sets for algebras of type $C_n$ is
\[
	\binom{3n-2}{n-1}.
\]
\end{cor}

\begin{proof}
We follow the proof of the previous corollary but use Theorem \ref{thm A} instead of Theorem \ref{thm B}. We need to take one of the $\frac{(3n-3)!}{(n-1)!(2n-1)!}$ unordered unoriented chord diagrams with $n-1$ chords in a circle with $n$ marked points. This time we add a loop at one of the marked points. We also need to place the loop in one of the regions which abut the marked point in order to specify the central region so that the chords will become oriented. To count the number of possibilities, we start with the $n$ marked points. Each chord cuts two of the angles and creates 2 new places. So there are $n+2(n-1)=3n-2$ choices giving
\[
	(3n-2)\frac{(3n-3)!}{(n-1)!(2n-1)!}=\binom{3n-2}{n-1}
\]
unordered pointed chord diagrams which are in bijection with exceptional sets.
\end{proof}

\begin{rem}\label{rem: about APR and orientation}
It follows from APR-tilting \cite{APR} that the number of exceptional sequences and exceptional sets is independent of the orientation of the quiver. This is because $(X,Y)$ is an exceptional pair in $mod\text-\Lambda$ for $\Lambda$ hereditary if and only if, in the bounded derived category,
\[
	\Hom(Y[a],X[b])=0
\]
for all integers $a,b$. Thus, an exceptional sequence of modules is equivalent to an exceptional sequence of orbits under the shift operator $[1]$ of objects in the bounded derived category which is independent of orientation of the quiver.

However, different orientations of the quiver might change the combinatorial model for exceptional sequences and also may change which objects are relatively projective and injective. So, orientation cannot be ignored.
\end{rem}

\section{Signed exceptional sequences in the tube $\cW_n$}\label{sec: signed exc seq in Wn}

In this section we prove Theorem \ref{thm F} which determines the probability distribution of relative projectives in a complete exceptional sequence for the tube $\cW_n$ and Lemma \ref{lem D} which gives the expected correspondence between signed exceptional sequences in $\cW_n$ and ordered rigid objects.

\subsection{Proof of Theorem \ref{thm F}}\label{ss: proof of Theorem G}

We recall the analogous result for type $A_n$.

\begin{thm}\cite{Ig-Sen}\label{thm: forests in type An}
There is a bijection
\[
	\left\{
	\begin{matrix}\text{complete exceptional sequences}\\
	E_\ast=(E_1,\cdots,E_n)\\
	\text{for linearly ordered quiver of type $A_n$}
	\end{matrix}
	\right\}
	\xrightarrow\cong
	\left\{
	\begin{matrix}\text{rooted labeled forests}\\
	\text{$F$ with $n$ vertices}
	\end{matrix}
	\right\}
\]
 Furthermore, $E_j$ is relatively projective in $E_\ast$ if and only if the $j$th vertex $v_j$ of the corresponding forest $F$ is either a descending vertex or a root of $F$. Similarly, $E_j$ is relatively injective in $E_\ast$ if and only if $v_j$ is either an ascending vertex or a root of $F$.
\end{thm}

We show that a similar statement holds for $\cW_n$.

\begin{thm}\label{thm: relative projectives in Wn}
Let $V_\ast=(V_1,\cdots,V_{n-1})$ be a complete exceptional sequence for the tube of rank $n$. Let $(F,\varepsilon)$ be the corresponding augmented rooted labeled forest. Then $V_j$ is relatively projective in $V_\ast$ if and only if the $j$th vertex $v_j$ of $F$ is a descending vertex. Similarly, $V_j$ is relatively injective if and only if $v_j$ is an ascending vertex of $F$.
\end{thm}

We show only the first statement. The second statement follows by duality. 

A rooted labeled forest has three kinds of vertices: roots, ascending vertices and descending vertices. We will show that the objects in $V_\ast$ corresponding to roots and ascending vertices are never relatively projective while the third kind are always relatively projective. This will prove the theorem.

\begin{rem}\label{rem: discrepancy with Chen-Igusa theorem} By Chen and Igusa \cite{CI}, every object in a complete exceptional sequence over a hereditary algebra is either relatively injective or relatively projective. Theorem \ref{thm: relative projectives in Wn} implies that, in the abelian tube $\cW_n$, a complete exceptional sequence will always have objects which are neither relatively projective nor relatively injective since these corresponding to the roots of the labeled forest and every forest has at least one root. The discrepancy is explained as follows. Any tube $\cW_n$ is a wide subcategory in $mod\text-\Lambda$ for some tame hereditary algebra $\Lambda$ which is not uniquely determined.
To determine whether an object in an exceptional sequence$(V_1,\cdots,V_{n-1})$  in the tube $\cW_n$ is relatively projective, we add terms to the left of $V_1$ to form a complete exceptional sequence for $\Lambda$. Then $V_k$ is relatively projective n $\cW_n$ if it is relatively projective in the completed exceptional sequence for $mod\text-\Lambda$. The proofs of Lemmas \ref{lem: roots are not projective in tubes}, \ref{lem: rel proj in tube implies descending}, \ref{lem: descending vertices are relatively projective} below imply that this criterion does not depend of the choice of $\Lambda$! The analogous statement for relatively injective objects uses the dual argument which requires completing the exceptional sequence by adding terms to the right of $V_{n-1}$. Completing on the left and right does not give the same complete exceptional sequence for $\Lambda$ which is why the result of \cite{CI} does not apply.
\end{rem}

\begin{lem}\label{lem: roots are not projective in tubes}
If $v_k$ is a root of $F$ then $V_k$ is not relatively projective in the tube.
\end{lem}

\begin{proof}
Let $v_{k_1}, \cdots, v_{k_r}$ be the roots of $F$ with $k_1<k_2<\cdots<k_r$. The corresponding objects $V_{k_i}$ are in reverse order in the Auslander-Reiten sequence. See Figure \ref{fig: roots are not projective in tubes.}. The indices are decreasing from left to right since each $V_{k_j}$ extends $V_{k_{j+1}}$. Label the quasi-simple objects at the mouth of the tube $0,1,2,\cdots,n-1$ starting with the unique simple object $S_0$ which is not in the support of any object in the exceptional sequence. Let $X_0,X_1$, etc. be the objects on the ray (blue in Figure \ref{fig: roots are not projective in tubes.}) at the simple $S_0$ with $X_0=S_0$ and $X_1,X_2,\cdots$ the objects on the blue line which are also on the same coray (slope $-1$ line) as $V_{k_{r}},V_{k_{r-1}},\cdots$, resp. See Figure \ref{fig: roots are not projective in tubes.}. Looking at the position of $X_i$, we see that $\Ext^1(V_{k_{r-i}},X_i)\neq 0$. 

Also, $X_i$ is not in the right $hom$-$ext$ perpendicular category of $\bigoplus V_j$ for $j>k_{r-i}$. This is because (1) $\Hom(V_j,X)=0$ for all $V_j$ in the exceptional sequence and all objects $X$ on the ray ascending from $S_0=X_0$ (blue in Figure \ref{fig: roots are not projective in tubes.}) and (2) $\Hom(X_i,\tau V_j)\neq0$ only for $V_j$ on the same ray as $V_{k_{r-i}}$. But any such $V_j$ maps nontrivially to $V_{k_{r-i}}$. So, we must have $j\le{k_{r-i}}$. Therefore, $X_i$ lies in the right $hom$-$ext$ perpendicular category of $V_j$ for all $j>{k_{r-i}}$. Therefore, $V_{k_{r-i}}$ is not relatively projective in the tube.
\end{proof}
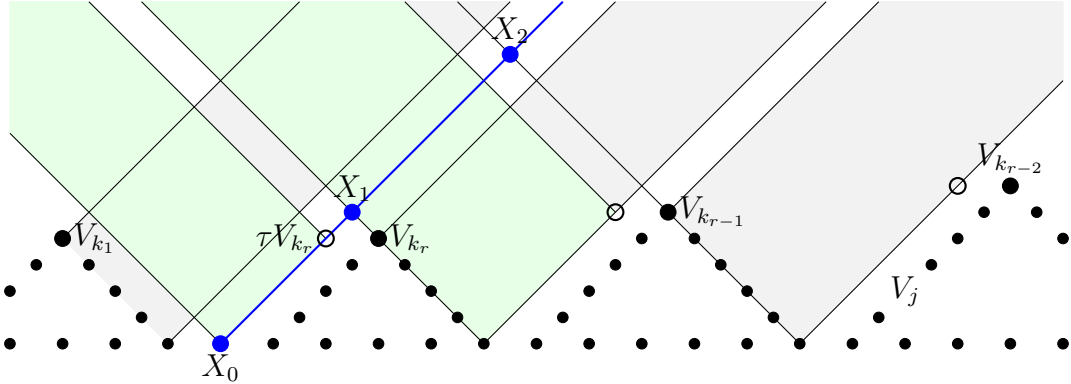
\begin{figure}[htbp]
\begin{center}
\begin{tikzpicture}[scale=.6]
\coordinate (X0) at (0,0);
\coordinate (X1) at (2.5,2.5);
\coordinate (X2) at (5.5,5.5);
\coordinate (X2p) at (6.5,6.5);
\coordinate (En) at (-3,2);
\coordinate (Enp) at (1.5,6.5);
\coordinate (E0) at (3,2);
\coordinate (E0p) at (7.5,6.5);
\coordinate (E00) at (2,2);
\coordinate (E00p) at (-2.5,6.5);
\coordinate (E1) at (8.5,2.5);
\coordinate (E11) at (7.5,2.5);
\coordinate (E11p) at (3.5,6.5);
\coordinate (E1p) at (12.5,6.5);
\coordinate (E2) at (15,3);
\coordinate (E22) at (14,3);
\coordinate (Z) at (13,1);
\draw[fill,gray!10!white](Enp)--(En)--(-1,0)--(5.5,6.5)--(Enp);
\draw[fill,gray!10!white](E0p)--(E0)--(5,0)--(11.5,6.5)--(E0p);
\draw[fill,gray!10!white](E1p)--(E1)--(11,0)--(16,5)--(16,6.5)--(E1p);
\draw[fill,green!10!white](E00p)--(E00)--(0,0)--(-4,4)--(-4,6.5)--(E00p);
\draw[fill,green!10!white](E11p)--(E11)--(5,0)--(-1.5,6.5)--(E11p);
\draw (En)--(Enp) (-1,0)--(5.5,6.5);
\draw (E0)--(E0p) (5,0)--(11.5,6.5);
\draw (E1)--(E1p) (11,0)--(16,5);
\draw (E00)--(E00p) (0,0)--(-4,4);
\draw (E11)--(E11p) (5,0)--(-1.5,6.5) (11,0)--(4.5,6.5);
\draw[blue,thick] (X0)--(X2p);
\foreach \x in {-4,...,16}\draw[fill] (\x,0) circle[radius=1mm];
\foreach \x in {(-1.5,.5),(-2,1),(-2.5,1.5),(-3.5,1.5),(-4,1)}  
\draw[fill] \x circle[radius=1mm];
\foreach \x in {(3.5,1.5),(4,1),(4.5,.5),(2.5,1.5),(2,1),(1.5,.5)} 
\draw[fill] \x circle[radius=1mm];
\foreach \x in {(9,2),(9.5,1.5),(10,1),(10.5,.5), (8,2), (7.5,1.5), (7,1), (6.5,.5)} 
\draw[fill] \x circle[radius=1mm];
\foreach \x in {(12.5,.5), (13.5,1.5), (14,2), (14.5,2.5), (15.5,2.5),(16,2)} 
\draw[fill] \x circle[radius=1mm];

\draw[fill,blue] (X0) circle[radius=1.5mm];
\draw (X0) node[below]{$X_0$};
\draw[fill,blue] (X1) circle[radius=1.5mm];
\draw (X1) node[above]{$X_1$};
\draw[fill,blue] (X2) circle[radius=1.5mm];
\draw (X2) node[above]{$X_2$};
\draw[fill] (E0) circle[radius=1.5mm];
\draw (E0) node[right]{$V_{k_r}$};
\draw[thick] (E22) circle[radius=1.5mm];
\draw[thick] (E11) circle[radius=1.5mm];
\draw[thick] (E00) circle[radius=1.5mm];
\draw (E00) node[left]{$\tau V_{k_r}$};
\draw[fill] (En) circle[radius=1.5mm];
\draw (En) node[right]{$V_{k_1}$};
\draw[fill] (E1) circle[radius=1.5mm];
\draw[fill] (E2) circle[radius=1.5mm];
\draw (E1) node[right]{$V_{k_{r-1}}$};
\draw (E2) node[above]{$V_{k_{r-2}}$};
\draw (Z) node{$V_j$};
\end{tikzpicture}
\caption{Objects $V_{k_i}$ corresponding to the roots of $F$ are the tops of triangles containing the other objects of the exceptional sequence. Shaded gray are the supports of $\Hom(V_{k_i},-)$. In green are the supports of $\Ext^1(V_{k_i},-)$. $X_i$ is the object preventing $V_{k_{r-i}}$ from being relatively projective in the tube.}
\label{fig: roots are not projective in tubes.}
\end{center}
\end{figure}

\begin{lem}\label{lem: rel proj in tube implies descending}
If $v_q$ is an ascending vertex of $F$ then $V_q$ is not relatively projective in $\cW_n$.
\end{lem}

\begin{proof}
This follows from Theorem \ref{thm: forests in type An} which implies that, when $v_q$ is an ascending vertex of $F$, $V_q$ is not relatively projective in the $A_{n-1}$ quiver category and thus there is an object $X$ in the prependicular category of $V_j,j>q$ in the triangle so that $\Ext^1(V_q,X)\neq0$. Since the $A_{n-1}$ quiver category is exactly embedded in the tube, $\Ext^1(V_q,X)\neq0$ in the tube. 
\end{proof}


\begin{lem}\label{lem: descending vertices are relatively projective}
If $v_i$ is a descending vertex in $F$ then $X_i$ is a relatively projective object in the exceptional sequence.
\end{lem}

\begin{proof} Suppose that $v_i$ is child of $v_k$ where $i<k$ and $v_{j_p}$ are all of the children of $v_k$ which come after $v_i$ where
\[
	i=j_s<j_{s-1}<\cdots<j_1<k.
\]
 Then we show that $X_i$ is relatively projective. We assume that $\cW_n$ is embedded in the module category of some hereditary algebra $\Lambda$ and that $\tau$, acting on $\cW_n$, is Auslander-Reiten translation in $mod\text-\Lambda$.

\underline{Claim 1}. $X_{j_1}$ is relatively projective in $X_k^\perp$.

Pf: Since $v_{j_1}$ is the last child in the forest that comes before $v_k$, $X_{j_1}$ is a submodule of $X_k$. So, $X_{j_1}=V_{ab}$ and $X_k=V_{ac}$ for some $a<b<c$. Then $\tau X_{j_1}=V_{a-1,b-1}$ is a submodule of $\tau X_k=V_{a-1,c-1}$. This is what we need. Suppose that $P$ is the projective cover of $X_{j_1}$ in $X_k^\perp$. Since $X_k^\perp$ is an abelian category exactly embedded in $mod\text-\Lambda$, it contains the kernel $K$ of $P\onto X_{j_1}$. But, the extension $K\to P\to X_{j_1}$ gives an element of $\Hom_\Lambda(K,\tau X_{j_1})$. Since $\tau X_{j_1}\subset \tau X_k$, we have $\Hom_\Lambda(K,\tau X_k)\neq0$ contradicting that $K\in X_k^\perp$. Therefore $K=0$ and $P=X_{j_1}$ is projective in $X_k^\perp$.

\underline{Claim 2}. There is a filtration
\[
	0=B_0\subset B_1\subset B_2\subset\cdots\subset B_s\subset X_k
\]
so that (a) $B_t/B_{t-1}\cong X_{j_t}$ for each $t$ and (b) $B_t$ is a projective object of $X_k^\perp$ for every $t$.

Pf: (b) is analogous to Claim 1. (a) follows from the following description of the modules $X_{j_t}$: Each $X_{j_t}=V_{a_t,a_{t+1}}$. So, we let $B_t=V_{a_1,a_{t+1}}$ and (a) will be satisfied.

We now show that $X_i=X_{j_s}$ is relatively projective. More precisely, we show that it is a projective object of $Z^\perp$ where $Z=X_k\oplus X_{j_1}\oplus X_{j_2}\oplus \cdots\oplus X_{j_{s-1}}$. To show this, suppose not. Let $P$ be the projective cover of $X_i$ in $Z^\perp$. Since $B_s$ is a projective object of $X_k^\perp$ which contains $Z^\perp$, the epimorphism $f: B_s\onto X_i$ lifts to $\tilde f:B_s\to P$. However, the kernel of $f$ is $B_{s-1}$ which has a filtration with subquotients $X_{j_1},\cdots,X_{j_{s-1}}$. Therefore, $\Hom(B_{s-1},P)=0$. So, $\tilde f:B_s\to P$ factors through $X_i$ making $X_i$ a direct summand of $P$ and therefore a projective object in $Z^\perp$.
\end{proof}

These three lemmas complete the proof of Theorem \ref{thm: relative projectives in Wn}. The probability distribution of relative projective objects in a complete exceptional sequence in the tube $\cW_n$ now follows directly from the following result.

\begin{thm}\cite[Corollary 5.2]{Ig-Sen}\label{thm: probabilities for descending vertices}
Let $F$ be a random rooted labeled forest with $n$ vertices. Then, the probability that $v_j$ is a descending vertex is
\[
	\mathbb P(v_j\text{ is descending})=\frac{n-j}{n+1}.
\] 
Furthermore, these events, for distinct $j$, are independent.
\end{thm}

Combining Theorems \ref{thm: relative projectives in Wn}
and \ref{thm: probabilities for descending vertices} (with $n$ replaced by $n-1$), we obtain the following.

\begin{cor}\label{cor: second part of thm F}
Let $(V_1,\cdots,V_{n-1})$ be a random complete exceptional sequence in $\cW_n$. Then, the probability that $V_j$ is relatively projective is
\[
	\mathbb P(V_j\text{ is relatively projective})=\frac{n-j-1}{n}.
\] 
Furthermore, these events, for distinct $j$, are independent.
\end{cor}

Dually, and since no vertex of a forest can be both ascending and descending, we have the following.

\begin{cor}\label{cor: dual of second part of F}
The probability that $V_j$ is relatively injective is
\[
	\mathbb P(V_j\text{ is relatively injective})=\frac{j-1}{n}.
\] 
Furthermore, these events, for distinct $j$, are independent. However, relative injectivity and relative projectivity are not independent since $V_j$ will never be both relatively injective and relatively projective.
\end{cor}

\subsection{Proof of Lemma \ref{lem D}}\label{ss: proof of Lemma E}

We review the main results of \cite{IT13} about cluster morphisms.

For $\cA$ a finitely generated wide subcategory of $mod\text-\Lambda$ for $\Lambda$ hereditary, let $\cC_\cA$ be the set consisting of (isomorphism classes of) indecomposable objects of the cluster category of $\cA$, i.e., indecomposable rigid objects of $\cA$ and shifted indecomposable projective objects of $\cA$: $P[1]$. A finite subset $T=\{T_1,\cdots,T_k\}$ of $\cC_\cA$ is called a \emph{partial cluster tilting set} if the $T_i$ are pairwise ext-orthogonal.

Let $\cB$ be a finitely generated wide subcategory of $\cA$. Then a \emph{cluster morphism} 
\[
[T]:\cA\to \cB
\]
 is defined to be a partial cluster tilting set $T$ in $\cC_\cA$ so that 
\[
	\cB=\cA\cap |T|^\perp
\]
where $|T|=\bigoplus |T_i|$ and $|T_i|\in\cA$ is the underlying object of $T_i\in \cA\cup \cA[1]$. We say $[T]$ has \emph{length} $k$ if $T$ has $k$ elements.
 
A \emph{signed exceptional sequence} $(X_1,\cdots,X_k)$ of length $k$ in $\cA$ is defined to be a sequence of $k$ composable cluster morphisms of length 1 starting at $\cA$:
\[
	\cB_0\xleftarrow{[X_1]} \cB_1\xleftarrow{[X_2]} \cdots \xleftarrow{[X_k]}  \cB_k=\cA.
\]Thus each $X_i\in\cC_{\cB_i}\subset \cA\cup \cA[1]$ and this definition is equivalent to the earlier definition of a signed exceptional sequence, namely, $(|X_1|,\cdots,|X_k|)$ is an exceptional sequence in $\cA$ with relatively projective objects allowed to be shifted.

The composition of these as morphisms in the cluster morphism category is the cluster morphism $[T]:\cA\to \cB_0$ where $T=\{T_1,\cdots,T_k\}\subset \cC_\cA$ is the unique partial cluster tilting set so that, for each $i$, $\undim T_i-\undim X_i$ is a $\mathbb Z$-linear combination of the vectors $\undim X_j$ for $j>i$. This implies that
\[
	\cB_0=\cA\cap |T|^\perp=\cA\cap |X|^\perp.
\]
The main theorem about cluster morphisms is the following.

\begin{thm}\cite{IT13}\label{thm: cluster morphism bijection}
This construction gives a bijection between signed exceptional sequences of length $k$ in $\cA$ and ordered partial cluster tilting set in $\cC_\cA$ of size $k$.
\end{thm}

The following special case of this bijection was Lemma \ref{lem D} in the introduction.

\begin{thm}\label{thm: Lemma E}
Let $\cW$ be any wide subcategory of $mod\text-\Lambda$ which contains no projective $\Lambda$-modules. Then, there is a bijection:
\[
	\left\{
	\begin{matrix}\text{signed exceptional sequences}\\
	\text{of length $k$ in $\cW$}
	\end{matrix}
	\right\}
	\xrightarrow[\cong]{\theta_k}
	\left\{
	\begin{matrix}\text{sequences of $k$ ext-orthogonal}\\
	\text{rigid objects in $\cW$}
	\end{matrix}
	\right\}
\]
\end{thm}

\begin{proof}[Proof of Theorem \ref{thm: cluster morphism bijection}]
Let $\cA=mod\text-\Lambda$ and consider the bijection
\[
	\theta_k(X_1,\cdots,X_k)=(T_1,\cdots,T_k)
\]
from Theorem \ref{thm: cluster morphism bijection}. If $|X_i|\in \cW$ for all $i$ and $\cB=|X|^\perp$, then $\cB$ is finitely generated, being a perpendicular category, and $\,^\perp \cB$ is the smallest wide subcategory of $mod\text-\Lambda$ which contains all $|X_i|$. Thus $\,^\perp \cB\subset\cW$. So, $|T_i|\in \,^\perp\cB\subset \cW$. Since $\cW$ contains no projective $\Lambda$-modules, $T_i$ is an object of $\cW$, not a shifted projective. So, $(T_1,\cdots,T_k)$ is a $k$-tuple of ext-orthogonal rigid objects of $\cW$.

Conversely, suppose $T_i\in \cW$ for all $i$. Then $\cB=|X|^\perp=|T|^\perp$ has left perpendicular category $\,^\perp\cB\subset \cW$. So, $|X_i|\in \,^\perp\cB\subset\cW$ for all $i$ as claimed.
\end{proof}

\begin{cor}\label{cor: thm D}
There is a bijection between signed exceptional sequences of length $k$ in the tube $\cW_n$ and signed exceptional sequences of length $k$ in module categories of type $B_{n-1}$ or $C_{n-1}$ (with any orientation of the arrows).
\end{cor}

\begin{proof}Since $\cW$ has no projective objects we have, by Lemma \ref{lem D} (Theorem \ref{thm: Lemma E}), a bijection between signed exceptional sequences of length $k$ in $\cW$ and ordered $k$-tuples of ext-orthogonal objects of $\cW$. By \cite{BMV}, these are in bijection with ordered partial cluster tilting sets of size $k$ for $mod\text-\Lambda$ if $\Lambda$ is hereditary of type $B_{n-1}$ or $C_{n-1}$. By Theorem \ref{thm: cluster morphism bijection}, these are in bijection with signed exceptional sequences of length $k$ in $mod\text-\Lambda$.
\end{proof}

\subsection{Bijection between signed exceptional sequences in $\cW_{n+1}$ and in $mod\text-\Lambda$ for $\Lambda$ of type $C_{n}$}\label{ss: bijection signed exc seq Wn and Cn-1}

We increase the index $n$ by 1 for convenience.
Going through the proof of Theorem \ref{thm C} (Corollary \ref{cor: thm D} above), the bijection is given by composing the following three bijections for $1\le k\le n$. We also add a fourth mapping $\alpha$ which is a monomorphism, not a bijection. The last set has $2^kn^k\binom nk$ elements since we allow each of the $k$ terms in each of the $n^k\binom nk$ soft exceptional sequence to have any sign.
\[
\xymatrix{
(V_1,\cdots,V_k)  & \{\text{length $k$ signed exceptional sequences in $\cW_{n+1}$}\} \ar[d]^{\theta_k}_\cong\\
(R_1,\cdots,R_k)  & \{\text{$k$-tuple of ext-orthogonal rigid objects in $\cW_{n+1}$}\} \ar[d]^{\beta_k}_\cong\\
(T_1,\cdots,T_k) & \{\text{ordered partial cluster of size $k$ in $mod\text-\Lambda$}\}\ar[d]^{\chi_k}_\cong\\
(X_1,\cdots,X_k)  & \{\text{length $k$ signed exceptional sequence in $mod\text-\Lambda$}\}\ar[d]^{\alpha_k}_\cap\\
(W_1,\cdots,W_k)  & \{\text{length $k$ soft exceptional sequences in $\cW_{n}$ with signs}\} 
	}
\]
By definition, the three bijections are recursive in the sense that, for $k\ge2$, dropping the first terms $V_1,R_1,T_1,X_1$ gives the bijections for $k-1$:
\[
	(V_2,\cdots,V_k)\xrightarrow{\theta_{k-1}} (R_2,\cdots,R_k)\xrightarrow{\beta_{k-1}} (T_2,\cdots,T_k)\xleftarrow{\theta_{k-1}}(X_2,\cdots,X_k) 
\]
Both $\theta_k$ and $\chi_k$ are given by linear twist equations (from \cite{IT13}): For $k=1$, we have the identity: $\theta_1=id$, $\chi_1=id$. (Thus, for any $k$, $R_k=V_k$, $X_k=T_k$.) 

For $k\ge2$, if $\theta_{k-1}(V_2,\cdots,V_k)=(R_2,\cdots,R_k)$, then $R_1$ is the unique rigid object of $\cW_{n+1}$ which is ext-orthogonal to $R_2,\cdots,R_k$ and whose dimension vector $\undim R_1\in \ZZ^{n+1}$ is congruent to $\undim V_1$ modulo the span of $\undim V_j$ for $2\le j\le k$. By Lemma \ref{lem D} (Theorem \ref{thm: Lemma E}), there is a unique positive vector $\undim R_1$ fitting this description.

Similarly, given $\chi_{k-1}(T_2,\cdots,T_k)=(X_2,\cdots,X_k)$, $X_1$ is the unique indecomposable object of $mod\text-\Lambda\coprod mod\text-\Lambda[1]$ so that $|X_1|$, the underlying module of $X_1$, extends the given exceptional sequence to $(|X_1|,|X_2|,\cdots,|X_k|)$ and whose dimension vector, $\undim X_1=(\text{sign} X_1)\undim |X_1|\in\ZZ^n$, is congruent to $\undim T_1$ modulo the span of $\undim X_j$ for $2\le j\le k$. If this formula gives a negative vector for $\undim X_1$ then the theorem is that $|X_1|$ is relatively projective in the exceptional sequence $(|X_1|,|X_2|,\cdots,|X_k|)$.

The bijection $\beta_k$ is from \cite{BMV}. It is given term-by-term by choosing an equivalence between the Auslander-Reiten quiver of the cluster category of $\Lambda$ and the ``mouth of the tube'' which is the portion of the Auslander-Reiten quiver of $\cW_{n+1}$ consisting of the rigid objects. Up to isomorphism, the cluster category of $\Lambda$, for $\Lambda$ of type $C_{n}$, is independent of the orientation of the quiver of $\Lambda$. So, we take it to be the straight descending orientation:
\[
	1\leftarrow 2\leftarrow \cdots\leftarrow {n}
\]
with the long root at the last vertex. We take one of the standard models:
\[
	\RR\leftarrow \RR\leftarrow \cdots\leftarrow\RR\leftarrow \CC.
\]
Then we have a bijection $\beta$ from the set of rigid objects of $\cW_{n+1}$, which are $V_{ij}$ for distinct $0\le i,j \le n$, to the set of bricks in $\cW_n$ with certain signs allowed. The bricks are $W_{ij}$ with not necessarily distinct $ i,j$ taken modulo $n$. We will obtain negative signs only for $W_{nj}=W_{0j}$. Thus $W_{0j}[1]$ will be in the image of $\beta$. The formula for $\beta$ will be:
\[
	\beta(V_{ij})=\begin{cases} W_{ij} & \text{if } 0\le i<j\le n\\
   W_{0,j+1}[1] & \text{if } i=n\\
   W_{i,j+1} & \text{otherwise}
    \end{cases}
\]
See Figure \ref{fig: W4 to W3} for the case $n=3$.

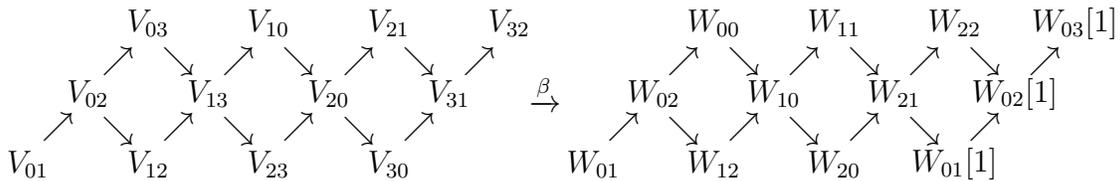
\begin{figure}[htbp]
\begin{center}
\begin{tikzpicture}[xscale=1.6,yscale=1.25]
\begin{scope}
\foreach \x/\y in {(1,0)/01,(2,0)/12,(3,0)/23,(4,0)/30}
\draw \x node{$V_{\y}$};
\foreach \x/\y in {(1.5,.75)/02,(2.5,0.75)/13,(3.5,0.75)/20,(4.5,0.75)/31}
\draw \x node{$V_{\y}$};
\foreach \x/\y in {(2,1.5)/03,(3,1.5)/10,(4,1.5)/21,(5,1.5)/32}
\draw \x node{$V_{\y}$};
\foreach \x in {1,...,4}\draw (\x.25,.4) node{$\nearrow$};
\foreach \x in {1,...,4}\draw (\x.75,1.1) node{$\nearrow$};
\foreach \x in {1,...,3}\draw (\x.75,.35) node{$\searrow$};
\foreach \x in {2,...,4}\draw (\x.25,1.1) node{$\searrow$};
\end{scope}
\draw (5.3,.75) node{$\xrightarrow\beta$};
\begin{scope}[xshift=4.7cm]
\foreach \x/\y in {(1,0)/01,(2,0)/12,(3,0)/20}
\draw \x node{$W_{\y}$};
\draw (4,0) node{$W_{01}[1]$};
\draw (4.5,0.75) node{$W_{02}[1]$};
\draw (5,1.5) node{$W_{03}[1]$};
\foreach \x/\y in {(1.5,.75)/02,(2.5,0.75)/10,(3.5,0.75)/21}
\draw \x node{$W_{\y}$};
\foreach \x/\y in {(2,1.5)/00,(3,1.5)/11,(4,1.5)/22}
\draw \x node{$W_{\y}$};
\foreach \x in {1,...,4}\draw (\x.25,.4) node{$\nearrow$};
\foreach \x in {1,...,4}\draw (\x.75,1.1) node{$\nearrow$};
\foreach \x in {1,2,3}\draw (\x.75,.35) node{$\searrow$};
\foreach \x in {2,3,4}\draw (\x.25,1.1) node{$\searrow$};
\end{scope}
\end{tikzpicture}
\caption{$\beta$ maps $V_{ij}$ to $W_{ij}$ for $i<j$ and $\beta(V_{ij})=W_{i,j+1}$ for $i>j$ except that $\beta(V_{nj})=W_{0,j+1}[1]$ where the indices for $W$ are modulo $n=3$.}
\label{fig: W4 to W3}
\end{center}
\end{figure}
The objects $W_{ij}$ correspond to rigid indecomposables in $mod\text-\Lambda$ for $\Lambda$ of type $C_n$. For example, $W_{0j}$ corresponds to $P_j$, the $j$th projective object, $W_{0j}[1]$ is the shifted projective and $W_{n-1,j}$ is the $j+1$st injective. Signed exceptional sequences in $mod\text-\Lambda$ will have other shifted objects, not just shifted projectives, but we use the notation $W_{ij}[1]$ for those.

Here is an example with $k=n=3$.

\begin{eg}
Take the signed exceptional sequence $(V_{12}[1],V_{13},V_{01})$. $V_{12}$ is relatively projective since, in the forest, the first vertex is descending:
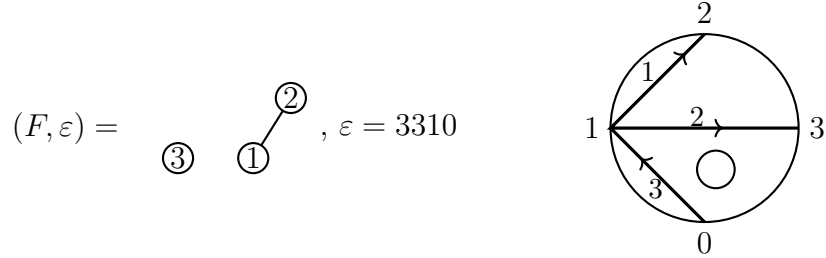
\begin{center}
\begin{tikzpicture}
\coordinate (T) at (-1.5,.4);
\coordinate (E) at (2.8,.4);
\coordinate (A) at (0,0);
\coordinate (C) at (1.5,.8);
\coordinate (D) at (1,0);
\begin{scope}
\draw (T) node{$(F,\varepsilon)=$};
\draw (E) node{,  $\varepsilon=3310$};
\draw[thick] (A) (C)--(D);
	\foreach \x/\y in {A/1,C/3,D/2}
	\draw[fill,white] (\x) circle[radius=2mm];
	\foreach \x/\y in {A/3,C/2,D/1}
	\draw[thick] (\x) circle[radius=2mm] node{$\y$};
\end{scope}
\begin{scope} [xshift=7cm,yshift=4mm,scale=.5] 
\draw[thick] (0,0) circle[radius=25mm];
\draw[very thick] (2.5,0)--(-2.5,0)--(0,2.5);
\draw[very thick] (0,-2.5)--(-2.5,0);
\draw[ thick,->] (-2.5,0)--(-.5,2);
\draw[ thick,->] (-2.5,0)--(.5,0);
\draw[ thick,->] (0,-2.5)--(-1.7,-.8);
\draw (0,2.5) node[above]{$2$};
\draw (-2.5,0) node[left]{$1$};
\draw (2.5,0) node[right]{$3$};
\draw (0,-2.5) node[below]{$0$};
\draw (-1.5,1.5) node{\small $1$};
\draw (-.2,.3) node{\small $2$};
\draw[thick] (.3,-1.1) circle[radius=5mm] ;
\draw (-1.3,-1.6) node{\small $3$};
\end{scope}
\end{tikzpicture}
\end{center}
The corresponding $3$-tuple of ext-orthogonal rigid objects in $\cW_4$ is $(V_{23},V_{03},V_{01})$.
\begin{enumerate}
\item The last object $V_{01}$ is always the same.
\item $V_{13}$ extends $V_{01}$ so we replace it with the extension $V_{03}$ with dimension vector
\[
	\undim V_{03}=\undim V_{13}+\undim V_{01}.
\]
\item $V_{12}$ extends $V_{01}$ so we replace $V_{12}[1]$ with $V_{23}$ with dimension vector
\[
	\undim V_{23}=\undim V_{03}-\undim V_{01}-\undim V_{12}.
\]
\item If $V_{12}$ were not shifted, we would replace it with $V_{02}$ with dimension vector
\[
	\undim V_{02}=\undim V_{12}+\undim V_{01}.
\]
\end{enumerate}
Since $\beta(V_{23},V_{03},V_{01})=(W_{20},W_{00},W_{01})$, the corresponding ordered cluster in $mod\text-\Lambda$ is $(I_1,P_3,P_1)$. This corresponds to the signed exceptional sequence $(X_1[1],X_3,P_1)$ where $X_1=W_{12}$, $X_3=W_{11}$ and $P_1=W_{01}$ in Figure \ref{fig: W4 to W3}.
\end{eg}

\begin{eg}\label{example in introduction}
For the example in the introduction as shown in Figure \ref{Fig: example of forest and directed chord diagram}:
\[
\begin{array}{ll}
(V_{30},V_{12}[1],V_{13}) & \text{a signed exceptional sequence in $\cW_4$ corresponds to}\\
(V_{10},V_{23},V_{13}) & \text{ext-orthogonal objects in $\cW_4$. Applying $\beta$ gives} \\
(W_{11},W_{20},W_{10}) & \text{in $\cW_3$ which corresponds to}\\
(X_3,I_1,X_2) & \text{in the cluster category of $mod\text-\Lambda$ for $\Lambda$ of type $C_3$ corresponding to}
\\
(P_3[1],X_1[1],X_2) & \text{a signed exceptional sequence in $mod\text-\Lambda$.}
\end{array}
\]
\end{eg}


\section{Probability distribution of relative projectives for $B_n/C_n$}\label{sec: signed exc seq for Cn}

Although this section is inspired by the correspondence between exceptional sequences of type $B_n$ and $C_n$ and augmented rooted labeled trees with $n$ vertices, our results about relatively projective objects in exceptional sequences of type $B_n/C_n$ do not use this correspondence. We also do not assume the quiver to be linearly oriented.

We will show that, in an exceptional sequence of length $k$: $(E_k,E_{k-1},\cdots,E_1)$ over a modulated quiver of type $B_n$ or $C_n$, the probability that $E_j$ is relatively projective, i.e., a projective object in the right perpendicular category, denoted $(E_{j-1}\oplus\cdots\oplus E_1)^\perp$, is equal to $\frac jn$ and these events are independent for different $j$. The independence for different $j$ will follow from the statement that the isomorphism class of the perpendicular category $(E_{j}\oplus\cdots\oplus E_1)^\perp$ is independent of whether or not $E_j$ is relatively projective. The proof is a simplified version of the proof in the $A_n$ case \cite{I:prob}.

\subsection{Counting subgraphs of a linear graph}\label{ss: counting subgraphs of linear graphs}
For any $m\ge0$, let $L_m$ denote the linear graph with vertices $0,1,2\cdots,m$ connected by $m$ edges $e_i$ connecting $i-1$ to $i$. The long root is at vertex 0. For example, $L_4$ is
\[
    L_4:\quad 0\,{\frac{e_1}\qquad }
\,1\,{\frac{e_2}\qquad }
\,2\,{\frac{e_3}\qquad }
\, 3\,{\frac{e_4}\qquad }
\, 4.
\]
For $\mu\ge0$ and $\lambda=(\lambda_0\le \lambda_1\le \cdots\le \lambda_k)$ a nonnegative partition of $n-\mu-k-1$ into $k+1$ parts $\lambda_i\ge0$, let $\cS_n(\mu;\lambda)$ denote the set of all subgraphs $G$ of $L_n$ consisting of $L_\mu$ and $k+1$ other linear subgraphs isomorphic to $L_{\lambda_i}$. Thus $G$ contains all the vertices of $L_n$ but is missing the edge $e_{\mu+1}$ and $k$ other edges $e_i$ where $\mu+1<i\le n$. For example,
\[
    G:\quad 0\,{\frac{e_1}\qquad }
\,1\,{\qquad }
\,2\,{\frac{e_3}\qquad }
\, 3\,{\qquad }
\, 4
\]
is an element of $\cS_4(1;(0,1))$ with $n=4, \mu=1, k=1, \lambda_0=0, \lambda_1=1$.

\begin{lem}
    The size of the set $\cS_n(\mu,\lambda)$ where $\lambda=(\lambda_0,\cdots,\lambda_k)$ is
    \[
    |\cS_n(\mu;\lambda)|={\frac{(k+1)!}{\prod n_p!}}
    \]
    where $n_p$ is the number of parts $\lambda_i=p$.
\end{lem}

\begin{proof}
    Elements of $\cS_n(\mu;\lambda)$ consist of $L_\mu$ followed by $k+1$ linear graphs $L_{\lambda_i}$ in some order. The number of these is $(k+1)!$ divided by the indicated redundancy factor.
\end{proof}

Let $\widetilde\cS_n(\mu;\lambda)$ be the set of all pairs $(G,e)$ where $G\in \cS_n(\mu;\lambda)$ and $e$ is one of the $n$ edges of $L_n$. Then we have:
\[
|\widetilde\cS_n(\mu;\lambda)|=n|\cS_n(\mu;\lambda)|.
\]

Let $\widetilde\cS_n^p(\mu;\lambda)$ be the set of all $(G,e)\in \widetilde\cS_n(\mu;\lambda)$ where $e$ is one of the $k+1$ edges of $L_n$ not in $G$. Then, either $e=e_{\mu+1}$ or $e=e_i$ is one of the $k$ other edges with $\mu+2\le i\le n$ not in $G$. We have:
\[
|\widetilde\cS_n^p(\mu;\lambda)|=(k+1)|\cS_n(\mu;\lambda)|
.
\]
Given $(G,e)\in \widetilde\cS_n^p(\mu;\lambda)$, consider what happens when the deleted edge $e$ is put back into $G$. Then we get $G'=G\cup\{e\}$ where either 
\begin{enumerate}
    \item $G'\in\cS_n(\mu';\lambda')$ where $\lambda'$ is $\lambda$ with one part $\lambda_i$ deleted and $\mu'=\mu+\lambda_i+1$ or
    \item $G'\in\cS_n(\mu;\lambda'')$ where $\lambda''$ is $\lambda$ with two parts $\lambda_i,\lambda_j$ deleted and one new part $\lambda_i+\lambda_j+1$ added.
\end{enumerate}
Given $G'$ as above, how may ways can we delete an edge $e$ from $G'$ to obtain an element of $\cS_n(\mu;\lambda)$ (with $G=G'\backslash e$ not necessarily the one we started with)?
\begin{enumerate}
    \item If $G'\in \cS_n(\mu';\lambda')$, there is only one possibility: $e=e_{\mu+1}$.
    \item If $G'\in\cS_n(\mu;\lambda'')$ we must first choose one of the $n_c''$ parts of $\lambda''$ equal to $c=\lambda_i+\lambda_j+1$. Then we have two cases depending on where $a=b$ or $a\neq b$ where $a=\lambda_i,b=\lambda_j$. If $a=b$ there is only one edge in $L_c$ which can be deleted to produce two subgraphs isomorphic to $L_a,L_b$. If $a\neq b$ there are two edges that can be deleted. Thus the number of possible edges $e$ that can be deleted from $G'$ to obtain an element of $\cS_n(\mu;\lambda)$ is $n_c''2X(a,b)$ where
    \[
    X(a,b)=\begin{cases}
        \frac12 &\text{ if } $a=b$\\
        1 &\text{ otherwise}.
    \end{cases}
    \]
\end{enumerate}

We obtain an equality from these two counts of the same set $|\widetilde\cS_n^p(\mu;\lambda)|$:
\begin{equation}\label{eq: recursive count of S(mu,l)}
    (k+1)|\cS_n(\mu;\lambda)|=\sum_{\lambda'}|\cS_n(\mu';\lambda')|+\sum_{\lambda''}n_c''2X(a,b)|\cS_n(\mu;\lambda'')|
\end{equation}
where the first sum is over all distinct $\lambda'$ equal to $\lambda$ with one term $a=\lambda_i$ deleted and $\mu'=\mu+a+1$ and the second sum is over all distinct $\lambda''$ obtained from $\lambda$ by deleting two parts $a=\lambda_i$ and $b=\lambda_j$ and adding the new part $c=a+b+1$. 

We summarize \eqref{eq: recursive count of S(mu,l)}: For each $\lambda'$, each element of $\cS_n(\mu';\lambda')$ gives one element of $\widetilde\cS_n^p(\mu;\lambda)$. For each $\lambda''$, each element of $\cS_n(\mu;\lambda'')$ gives $n_c''2X(a,b)$ elements of $\widetilde\cS_n^p(\mu;\lambda)$.

By counting the size of $\widetilde\cS_n(\mu;\lambda)$ in two ways we obtain another formula for $|\cS_n(\mu;\lambda)|$.
\begin{lem}\label{lem: second recursive count of S(mu,l)}
With the same notation as in \eqref{eq: recursive count of S(mu,l)} we have:
\[
n|\cS_n(\mu;\lambda)|=\sum_{\lambda'} (\mu+a+1)|\cS_n(\mu';\lambda')|+\sum_{\lambda''} n_c''(c+1)X(a,b) |\cS_n(\mu;\lambda'')|.
\]
\end{lem}

\begin{proof}
We construct a mapping $\varphi:\widetilde\cS_n(\mu;\lambda)\to \widetilde\cS_n^p(\mu;\lambda)$. Each term on the right hand side of \eqref{eq: recursive count of S(mu,l)} corresponds to an element of $\widetilde\cS_n^p(\mu;\lambda)$. If we multiply this with the size of its inverse image in $\widetilde\cS_n(\mu;\lambda)$ we will obtain the required formula for the size of $\widetilde\cS_n(\mu;\lambda)$. 

For any $G\in \cS_n(\mu;\lambda)$ let $G_0,G_1,\cdots,G_{k+1}$ be the components of $G$ in order and let $g_i\ge0$ be the number of edges in $G_i$. Thus, $g_0=\mu$ and the other $g_i$ are some permutation of the ${\lambda_j}$ in $\lambda$. Let $d_0,\cdots,d_k$ be the missing edges in order. Thus $d_i$ is the edge between $G_i$ and $G_{i+1}$. For any $(G,e)\in \widetilde\cS_n(\mu;\lambda)$, let $\varphi(G,e)$ be given as follows.
\begin{enumerate}
\item $\varphi(G,e)=(G,e)$ if $e\notin G$.
\item $\varphi(G,e)=(G,d_0)$ if $e\in G_0\cup G_1$. Thus $\varphi^{-1}(G,d_0)$ has $\mu+g_1+1$ elements.
\item $\varphi(G,e)=(G,d_i)$ if $e\in G_{i+1}$ and $i\ge2$. So, $|\varphi^{-1}(G,d_i)|=g_{i+1}+1$ if $i\ge2$.
\end{enumerate}

Now we count the number of elements of $\widetilde\cS_n(\mu;\lambda)$ corresponding to each element of the right hand side of \eqref{eq: recursive count of S(mu,l)}.
\begin{enumerate}
	\item Each element of $\cS_n(\mu';\lambda')$ gives an element $(G,d_0)\in \widetilde\cS_n^p(\mu;\lambda)$ with $g_0=\mu,g_1=a$. This has $\mu'=\mu+a+1$ inverse image points in $\widetilde\cS_n(\mu;\lambda)$.
	\item Take $G''\in\cS_n(\mu;\lambda'')$ where $a=b$ and $c=2a+1$. There are $n_c''$ components of $G''$ of size $c$ (excluding $G_0''$). In each such component, the middle edge $p$ is removed giving two new components of $G=G''\backslash p$ both of size $a$. Then $(G,p)$ is the corresponding element of $\widetilde \cS_n^p(\mu;\lambda)$ with $a+1$ inverse image points in $\widetilde \cS_n^p(\mu;\lambda)$. This gives
	\[
	(a+1)n_c''=n_c''(c+1)X(a,a)
	\]
elements of $\widetilde \cS_n(\mu;\lambda)$ since $X(a,a)=\frac12$ and $c+1=2a+2$.
	\item Take $G''\in\cS_n(\mu;\lambda'')$ where $a\neq b$ and $c=a+b+1$. There are $n_c''$ components of $G''$ of size $c$ (excluding $G_0''$). In each such component there are two edges, say $p,q$, which can be deleted to produce an element of $\cS_n(\mu;\lambda)$. In $G''\backslash p$, the $L_c$ component becomes $L_a\coprod L_b$ and in $G''\backslash q$, the $L_c$ becomes $L_b\coprod L_a$. In the first case the inverse image in $\widetilde\cS_n(\mu;\lambda)$ has $b+1$ elements, in the second case it has $a+1$ elements for a total of
	\[
		n_c''(a+1+b+1)=n_c''(c+1)X(a,b)
	\]
	inverse image points since $X(a,b)=1$.
\end{enumerate}
Adding these up gives the lemma.
\end{proof}

\subsection{Counting exceptional sequences of type $B_n$ or $C_n$}\label{ss: counting exc seq for Cn}

For $\mu\ge0$ and $\lambda=(\lambda_0,\cdots,\lambda_k)$ a nonnegative partition of $n-\mu-k-1$, let $\cN_n(\mu;\lambda)$ denote the set of all exceptional sequences $E_\ast=(E_{k+1},E_k,\cdots,E_1)$ for $B_n$ or $C_n$ whose perpendicular category $\cE=(E_{k+1}\oplus\cdots\oplus E_1)^\perp$ has type $B_\mu\times \prod A_{\lambda_i}$ and let $\cN_n^p(\mu;\lambda)$ be the subset of $\cN_n(\mu;\lambda)$ of exceptional sequences in which $E_{k+1}$ is relatively projective, i.e., a projective object of $\cE'=(E_k\oplus \cdots\oplus E_1)^\perp$. 

We will show that $|\cN_n^p(\mu;\lambda)|/|\cN_n(\mu;\lambda)|=\frac{k+1}n$ for all $\mu,\lambda$. For $k=0$, $\lambda=(\lambda_0)=(n-\mu-1)$ and $\cN_n(\mu;\lambda)$ consists of a single $\tau$-orbit, namely that of $P_{\mu+1}$, the $\mu+1$st projective object, and $\cN_n^p(\mu;\lambda)=\{P_{\mu+1}\}$ contains only that one element. Since every $\tau$-orbit contains $n$ elements including one projective object, for $k=0$ we have $|\cN_n(\mu;\lambda)|=n$ and $|\cN_n^p(\mu;\lambda)|=1$. This is the $k=0$ case of the following theorem.

\begin{thm}\label{thm: counting exc seq of type Bn/Cn}
    For $k\ge 0$, $\mu\ge0$ and $\lambda=(\lambda_0,\cdots,\lambda_k)$ a nonnegative partition of $n-\mu-k-1$ we have
    \[
|\cN_n(\mu,\lambda)|=n^{k}|\widetilde\cS_n(\mu,\lambda)|=n^{k+1}|\cS_n(\mu,\lambda)|=\frac{n^{k+1}(k+1)!}{\prod n_p!}
    \]
    \[
|\cN_n^p(\mu,\lambda)|=n^{k}|\widetilde\cS_n^p(\mu,\lambda)|=n^{k}(k+1)|\cS_n(\mu,\lambda)|=\frac{k+1}n|\cN_n(\mu,\lambda)|
    \]    
where $n_p$ is the number of parts $\lambda_i$ of $\lambda$ equal to $p$.
\end{thm}

This will follow almost immediately from the following lemma.

\begin{lem}\label{lem: recursive formula for Nn(m,l)} Using the notation in Equation \eqref{eq: recursive count of S(mu,l)} we have the following.
\[
|\cN_n(\mu;\lambda)|=\sum_{\lambda'} (\mu+a+1)|\cN_n(\mu';\lambda')|+\sum_{\lambda''} n_c''(c+1)X(a,b) |\cN_n(\mu;\lambda'')|.
\]
\[
|\cN_n^p(\mu;\lambda)|=\sum_{\lambda'} |\cN_n(\mu';\lambda')|+\sum_{\lambda''} n_c''2X(a,b) |\cN_n(\mu;\lambda'')|.
\]
\end{lem}

\begin{proof}
Given an exceptional sequence $E_\ast=(E_{k+1},E_k,\cdots,E_1)\in \cN_n(\mu;\lambda)$ we consider the shorter exceptional sequence $E_\ast'=(E_k,\cdots,E_1)$. There are two disjoint possibilities. Either $E_\ast'\in \cN_n(\mu';\lambda')$ or $E_\ast'\in \cN_n(\mu;\lambda'')$. We will see that there are $\mu+a+1$ possibilities for $E_{k+1}$ in the first case, one of which is relatively projective, and $n_c''(c+1)X(a,b)$ possibilities for $E_{k+1}$ in the second case, $2n_c''X(a,b)$ of which are relatively projective. This will prove both formulas in the lemma.

In more detail, take the first case $E_\ast'\in \cN_n(\mu';\lambda')$. The perpendicular category of $E_\ast'$ has type $B_{\mu'}\times \prod A_{\lambda_j'}$ and $E_{k+1}$ lies in $B_{\mu'}$ where $\mu'=\mu+a+1$. Furthermore, $E_{k+1}$ must lie in the $\tau$ orbit of the $\mu+1$st projective object of $mod\text-B_{\mu'}$ in order for $E_\ast=(E_{k+1},E_\ast')$ to lie in $\cN_n(\mu;\lambda)$. There are $\mu'=\mu+a+1$ objects in that $\tau$ orbit, one of which is projective. This gives the coefficients of $|\cN_n(\mu';\lambda')|$ in the two sums.

In the second case $E_\ast'\in \cN_n(\mu;\lambda'')$ where $\lambda''$ is $\lambda$ with parts of size $a,b$ removed and a new part of size $c=a+b+1$ added. $E_{k+1}$ must lie in one of the $n_c''$ copies of $A_c$ which occur in the perpendicular category of $E_\ast'$. When $a=b$, $E_{k+1}$ must lie in the $\tau$ orbit of the middle projective $P_{a+1}$ and there are $(c+1)/2=(c+1)X(a,a)$ object in that $\tau$-orbit. One of these is projective (and $1=2X(a,a)$). When $a\neq b$, $E_{k+1}$ must lie in one of two $\tau$-orbits, that of $P_{a+1}$ or $P_{b+1}$. The union of these has $a+b+2=(c+1)X(a,b)$ number of element. Two of these  $E_{k+1}$ are projective (one in each $\tau$-orbit), making $2X(a,b)$ projective elements. In both subcases of Case 2 there are $n_c''(c+1)X(a,b)$ choices of $E_{k+1}$ and $2n_c''X(a,b)$ of these are (relatively) projective. This gives the coefficients in both $\lambda''$ summands which finishes the proof of the lemma.
\end{proof}

\begin{proof}[Proof of Theorem \ref{thm: counting exc seq of type Bn/Cn}]
    By induction on $k$ we have that $|\cN_n(\mu',\lambda')|=n^k|\cS_n(\mu',\lambda')|$ and $|\cN_n(\mu,\lambda'')|=n^k|\cS_n(\mu,\lambda'')|$. Inserting these into the right hand side in Lemma \ref{lem: recursive formula for Nn(m,l)} we obtain $n^{k+1}|\cS_n(\mu,\lambda)|$ in the first equation by Lemma \ref{lem: second recursive count of S(mu,l)} and $(k+1)n^k|\cS_n(\mu,\lambda)|$ in the second equation by \eqref{eq: recursive count of S(mu,l)}. This proves both statements in the theorem.
\end{proof}

Theorem \ref{thm: counting exc seq of type Bn/Cn} implies the following with $k$ replacing $k+1$. 

\begin{cor}\label{cor: independence of rel projectivity}
In a random exceptional sequence of $\ell$ for $B_n$ or $C_n$, the probability that $E_k$ for $k\le \ell$ is relatively projective is equal to $k/n$ and this events is independent of the isomorphism class of the perpendicular category $(E_k\oplus\cdots\oplus E_1)^\perp$.
\end{cor}

\begin{proof}
    Each isomorphism class of perpendicular categories is given by a pair $(\mu,\lambda)$ and for each such pair the fraction of those $(E_k,\cdots,E_1)$ with that perpendicular category for which $E_k$ is relatively projective is $k/n$. So, the events are independent.
\end{proof}

\begin{cor}\label{cor: number of signed exc seq}
The events $D_k$ that $E_k$ is relatively projective are independent for distinct $k$.
\end{cor}

\begin{proof}
    For $j>k$, the event $D_j$ depends only of the isomorphism class of the perpendicular category $(E_k\oplus\cdots\oplus E_1)^\perp$ which is independent of $D_k$.
\end{proof}

Corollaries \ref{cor: independence of rel projectivity} and \ref{cor: number of signed exc seq} give Theorem \ref{thm E}. The next corollary completes the proof of Theorem \ref{thm A}.

\begin{cor}\label{cor: no of exc seq for Bn/Cn}
The number of exceptional sequences of length $k$ for $B_n$ or $C_n$ is $n^k\binom nk$.
\end{cor}

\begin{cor}\label{cor: no of signed exc seq for Bn/Cn}
The number of signed exceptional sequences of length $k$ for $B_n$ or $C_n$ is 
\[
(n+1)\cdots(n+k)\binom nk= \frac{(n+k)!}{k!(n-k)!}.
\]
\end{cor}

Since signed exceptional sequences are in bijection with ordered cluster we have:

\begin{cor}\label{cor: no of clusters for Bn/Cn}
The number of partial cluster of size $k$ for $B_n$ or $C_n$ is 
\[
\frac{(n+k)!}{k!k!(n-k)!}.
\]
\end{cor}

A \emph{partial cluster} of size $k$ is a rigid object in the cluster category having $k$ nonisomorphic indecomposable summands. For example, when $k=1$, this number is $n(n+1)=n^2+n$ which counts the $n^2$ indecomposable modules and the $n$ shifted projective modules $P_i[1]$. For $n=k=3$ we have $3^3=27$ exceptional sequence and $4\cdot 5\cdot 6= 120$ signed exceptional sequences giving $20$ clusters for $B_3$ and $C_3$

In terms of generating functions, the distribution of relative projectives in an exceptional sequence for $B_n$ or $C_n$ is given as follows.

For any hereditary algebra $\Lambda$, let $f_{\Lambda,k}(z_k,\cdots,z_1)$ be the $k$-variable generating function 
\[
	f_{\Lambda,k}(z_k,\cdots,z_1)=\sum_\beta a_\beta z^\beta
\]
where the sum is over all multi-indices $\beta=(b_k,\cdots,b_1)\in \{0,1\}^k$ and $a_\beta$ is the number of exceptional sequences $(E_k,\cdots, E_1)$ for $\Lambda$ for which $E_i$ is relatively projective when $b_i=1$ and not relatively projective for $b_i=0$.

\begin{cor}\label{cor: generating function for rel proj Bn,Cn}
The $k$-variable generating function for $\Lambda=B_n$ or $C_n$ is
\[
	f_{\Lambda,k}(z_k,\cdots,z_1)=\binom nk\prod_{i=1}^k (n-i+iz_i)
\]
\end{cor}

\begin{proof}
Since the probability of $E_i$ being relatively projective are independent of each other by Corollary \ref{cor: number of signed exc seq}, the generating function is given by multiplying the number of exceptional sequences by the product of the terms
\[
	1+(z_i-1)\PP(\text{$E_i$ is relatively projective})
\]
Since $\PP(\text{$E_i$ is relatively projective})=\frac in$ by Corollary \ref{cor: independence of rel projectivity} and the number of exceptional sequences is $n^k\binom nk$, the generating function is:
\[
	n^k\binom nk\prod_{i=1}^k \left(
	1+(z_i-1)\frac in
	\right)=\binom nk \prod_{i=1}^k(n+(z_i-1)i)
\]
as claimed.
\end{proof}

Corollary \ref{cor: no of exc seq for Bn/Cn} also follows from this by plugging in $z_i=2$ for all $i$.

\section*{Acknowledgements}
The authors would like to thank Gordana Todorov for her help and encouragement during the collaboration which lead to this paper. Also, the first author thanks Aslak Bakke Buan for discussions comparing cluster tubes and abelian tubes many years ago. He also thanks Bin Zhu for more recent discussions about cluster algebras and cluster tubes \cite{ZZ}. The first author also thanks Shujian Chen for very helpful discussions about chord diagrams. The first author is supported by Simons Foundation Grant \#686616.

We are thankful to the organizers of the 33rd meeting of Representation Theory of Algebras and Related Topics at University of Sherbrooke 2023, where these results were presented.

\end{document}